\documentclass[12pt]{amsart}
\usepackage{ amsmath, amsthm, amsfonts, amssymb, color}
 \usepackage{mathrsfs}
\usepackage{amsfonts, amsmath}
 \usepackage{amsmath,amstext,amsthm,amssymb,amsxtra}
 \usepackage{txfonts} 
 \usepackage[colorlinks, citecolor=blue,pagebackref,hypertexnames=false]{hyperref}
 \allowdisplaybreaks
 \usepackage{pgf,tikz}
 \usepackage{pst-plot}

\def\HHML { \mathbb{H}^1_{L, {mol}, M}(X) }

\begin{document}
\baselineskip 16pt

\newcommand\RR{\mathbb{R}}
\def\RN {\mathbb{R}^n}
\newcommand{\norm}[1]{\left\Vert#1\right\Vert}
\newcommand{\abs}[1]{\left\vert#1\right\vert}
\newcommand{\set}[1]{\left\{#1\right\}}
\newcommand{\Real}{\mathbb{R}}
\newcommand{\R}{\mathbb{R}}
\newcommand{\supp}{\operatorname{supp}}
\newcommand{\card}{\operatorname{card}}
\renewcommand{\L}{\mathcal{L}}
\renewcommand{\P}{\mathcal{P}}
\newcommand{\T}{\mathcal{T}}
\newcommand{\A}{\mathbb{A}}
\newcommand{\K}{\mathcal{K}}
\renewcommand{\S}{\mathcal{S}}
\newcommand{\Id}{\operatorname{I}}
\newcommand\wrt{\,{\rm d}}
\newcommand\Ad{\,{\rm Ad}}
\def\HR { H^1({\Bbb R}^n) }
\def\BL {{\rm BMO}_{L,M}(X)}
\def\HHAL { \mathbb{H}^1_{L,{at}, M}(X) }
\def\HAL { H^1_{L,{at}, M}(X) }
\def\HA{ H^1_{L,{at}, 1}(X) }
\def\HSL { H^1_{L }(X) }
\def\HML { H^1_{L, {mol}, M, \epsilon}(X) }
\def\HHML { \mathbb{H}^1_{L, {mol}, M, \epsilon}(X) }
\def\HM{ H^1_{L, {mol}, 1}(X) }
\def\HMH { H^1_{L, {\rm max}, h}(X) }
\def\HNH { H^1_{L, {\mathcal N}_h}(X) }

\def\HSP { H^1_{L, S_P}(X) }
\def\HMP { H^1_{L, {\rm max}, P}(X) }
\def\HNP { H^1_{L, {\mathcal N}_P}(X) }

\def\SL{\sqrt[m]L}
\newcommand{\mar}[1]{{\marginpar{\sffamily{\scriptsize
        #1}}}}
\newcommand{\li}[1]{{\mar{LY:#1}}}
\newcommand{\el}[1]{{\mar{EM:#1}}}
\newcommand{\as}[1]{{\mar{AS:#1}}}

\newcommand\CC{\mathbb{C}}
\newcommand\NN{\mathbb{N}}
\newcommand\ZZ{\mathbb{Z}}

\renewcommand\Re{\operatorname{Re}}
\renewcommand\Im{\operatorname{Im}}

\newcommand{\mc}{\mathcal}
\newcommand\D{\mathcal{D}}

\newtheorem{thm}{Theorem}[section]
\newtheorem{prop}[thm]{Proposition}
\newtheorem{cor}[thm]{Corollary}
\newtheorem{lem}[thm]{Lemma}
\newtheorem{lemma}[thm]{Lemma}
\newtheorem{exams}[thm]{Examples}
\theoremstyle{definition}
\newtheorem{definition}[thm]{Definition}
\newtheorem{rem}[thm]{Remark}

\numberwithin{equation}{section}
\newcommand\bchi{{\chi}}

\title[Sharp endpoint  estimates for Schr\"odinger  groups   ]
{Sharp endpoint estimates for Schr\"odinger  groups\\[1pt]  on   Hardy spaces
}

 \author  [P. Chen, X.T. Duong,   J. Li  and L. Yan]
 { Peng Chen, \, Xuan Thinh Duong,  \,  Ji Li \,  and \, Lixin Yan}

 \address{Peng Chen, Department of Mathematics, Sun Yat-sen
University, Guangzhou, 510275, P.R. China}
\email{chenpeng3@mail.sysu.edu.cn}

 \address{Xuan Thinh Duong, Department of Mathematics, Macquarie University, NSW 2109, Australia}
\email{xuan.duong@mq.edu.au}

\address{Department of Mathematics, Macquarie University, NSW, 2109, Australia}
\email{ji.li@mq.edu.au}

\address{Department of Mathematics, Sun Yat-sen (Zhongshan) University, Guangzhou, 510275, P.R. China}
\email{mcsylx@mail.sysu.edu.cn}

  \date{\today}

\begin{abstract}
Let $L$ be a non-negative self-adjoint operator acting on $L^2(X)$
where $X$ is a space of homogeneous type with a dimension $n$. Suppose that
the heat kernel of  $L$
satisfies  the Davies-Gaffney estimates of order $m\geq 2$.  Let $H^1_L(X)$ be the Hardy space associated
 with $L.$   In this paper
 we show  sharp endpoint estimate  for the Schr\"odinger group $e^{itL}$
 associated with $L$
  such that
 \begin{eqnarray*}
 \left\| (I+L)^{-{n/2}}e^{itL} f\right\|_{ L^1(X)}  +  \left\| (I+L)^{-{n/2}}e^{itL} f\right\|_{ H^1_L(X)}
 \leq  C(1+|t|)^{n/2}\|f\|_{H^1_L(X)}, \ \ \ t\in{\mathbb R}
\end{eqnarray*}
for some  constant $C=C(n, m)>0$ independent of $t$. By a duality and interpolation argument,
it    gives  a new  proof of
 a recent result of \cite{CDLY} for {  sharp} endpoint $L^p$-Sobolev bound    for  $e^{itL}$:
$$
 \left\| (I+L)^{-s }e^{itL} f\right\|_{ L^p(X)}  \leq C (1+|t|)^{s} \|f\|_{ L^p(X)}, \ \ \ t\in{\mathbb R}, \ \
  \ s\geq n\big|{1\over  2}-{1\over  p}\big|
$$
 for   every $1<p<\infty$ when the heat kernel of $L$ satisfies a  Gaussian upper bound,
which   extends  the classical  results due to  Miyachi  (\cite{Mi1, Mi}) for the Laplacian
on the Euclidean space ${\mathbb R}^n$.
 \end{abstract}

\maketitle


\section{Introduction}\label{sec:intro}
\setcounter{equation}{0}

Consider the Laplace operator $\Delta=-\sum_{i=1}^n\partial_{x_i}^2$ on the Euclidean space $\mathbb R^n$
and  the Schr\"odinger equation
 \begin{eqnarray*}\label{e1.00}
\left\{
\begin{array}{ll}
  i{\partial_t u } +\Delta u=0,\\[4pt]
 u|_{t=0}=f
\end{array}
\right.
\end{eqnarray*}
with initial data $f$. Its solution     can be written as
 \begin{eqnarray*}\label{e1.0}
 u(x, t)=e^{it\Delta} f(x)={1\over (2\pi)^n}\int_{{\mathbb R}^n} {\widehat f}(\xi) e^{i(  \langle x, \,  \xi\rangle +t|\xi|^2 )} d\xi,
 \end{eqnarray*}
 where ${\widehat f}$ denotes the Fourier transform of $f$.
It is well-known that the operator $e^{it\Delta}$ acts boundedly on $L^p({\mathbb R}^n)$ if and only if $p=2$;
 see  H\"ormander \cite{H1}. For $p\not= 2, $
it was shown (see for example,  \cite{Br, La, Sj})) that for $s > n|{1/ 2}-{1/p}|$, the operator
 $e^{it\Delta}$ maps the Sobolev space $L^p_{2s}(\RN)$ into $L^p(\RN)$, in other words,
$(I+\Delta)^{-s } e^{it\Delta}$ is bounded on $L^p(\RN)$.  For $s < n|{1/ 2}-{1/p}|$, it is known that the operator
$(I+\Delta)^{-s } e^{it\Delta}$ is unbounded on $L^p(\RN)$.
In \cite{Mi1},  Miyachi obtained the sharp endpoint  estimate for  $e^{it\Delta}$  on Hardy and Lebesgue spaces,
 and showed  that for every $0<p<\infty,$
\begin{eqnarray}\label{e1.1}
 \left\|  (1+\Delta)^{-s} e^{it\Delta} f\right\|_{H^p(\mathbb R^n)} \leq C  (1+|t|)^{s}\|f\|_{H^p (\mathbb R^n)},
  \ \ \ t\in{\mathbb R}, \ \ \ s\geq  n\big|{1\over  2}-{1\over  p}\big|,
\end{eqnarray}
where   $H^p({\mathbb R}^n)$ is  the classical Hardy space (\cite{FS}) on ${\mathbb R}^n$
  and $H^p({\mathbb R}^n)=L^p({\mathbb R}^n)$ if $1<p<\infty$.
  See also Fefferman-Stein's work \cite[Section 6]{FS}.

The Schr\"odinger semigroup $\{e^{it\Delta}\}_{t>0}$ can be defined in terms of the spectral resolution of
the self-adjoint Laplace operator $\Delta.$ A natural question is to determine a sufficient condition so that
\eqref{e1.1} holds when the Laplace $\Delta$ is replaced by a non-negative self-adjoint operator $L.$
  For this purpose we suppose  that $(X,d,\mu) $ is a metric measure space
 with a distance $d$ and a measure  $\mu$, and   $L$ is a non-negative self-adjoint operator on $L^2(X).$
  Such an operator $L$ admits a spectral
resolution
\begin{eqnarray*}
L=\int_0^{\infty} \lambda dE_L(\lambda),
\end{eqnarray*}
where  $E_L(\lambda)$ is the projection-valued measure supported on the spectrum of $L$.
The operator   $   e^{itL}, t\in{\mathbb R},$ is defined by
 \begin{equation}\label{e1.3}
 e^{itL}f =   \int_0^{\infty}    e^{it\lambda}dE_L(\lambda) f
   \end{equation}
 for $f\in L^2(X)$, and forms   the Schr\"odinger group.
   By the spectral theorem (\cite{Mc}),  the operator   $  e^{itL}$  is  continuous on $L^2(X)$.
   Our main interest will be in  the mapping properties
of families of operators derived from the Schr\"odinger group
on Hardy and Lebesgue spaces.

 Depending on the nature of the assumptions regarding the assumption of $e^{-tL}$, there
 are various nuances of  the mapping properties
of the Schr\"odinger group $e^{-tL}$
on $L^p$ spaces  presently available in the literature.
 For example, on Lie groups with polynomial growth and manifolds with non-negative Ricci curvature, similar results as in \eqref{e1.1}
for $s> n\left|{1/ 2}-{1/ p}\right|$ and $1<p<\infty$ have been first announced by Lohou\'e in \cite{Lo},
then obtained by Alexopoulos  in \cite{A}.
   In the abstract setting
of operators on metric measure spaces,   Carron, Coulhon and Ouhabaz \cite{CCO}  showed  that for every $1<p<\infty,$
\begin{eqnarray} \label{e1.5553}
 \left\| (I+L)^{-s }e^{itL} f\right\|_{p} \leq C (1+|t|)^{s} \|f\|_{p}, \ \ \ t\in{\mathbb R}, \ \
  \ s> n\big|{1\over  2}-{1\over  p}\big|,
\end{eqnarray}
  provided
the semigroup $e^{-tL}$, generated by $-L$ on $L^2(X)$,  has the kernel  $p_t(x,y)$
which  satisfies
the  Gaussian upper bound, i.e.
\begin{equation*}
 \label{GE}
 \tag{${\rm GE}_m$}
|p_t(x,y)| \leq {C\over V(x,t^{1/m})} \exp\left(-c \, {  \left({d(x,y)^{m}\over    t}\right)^{1\over m-1}}\right)
\end{equation*}
for every $t>0, x, y\in X$, where $c, C$ are   positive constants and $m\geq 2.$
Such estimate  \eqref{GE} is typical for elliptic or sub-elliptic differential operators of order $m$
(see for example, \cite{A,  CCO,   D, DM, DOS,  JN, JN2, O,  Si, Sj}
and the references therein). See also  related results in \cite{BDN, DN, FS, JN, JN2}.

The question whether  estimate \eqref{e1.5553} holds with  $s =  n|{1/2}-{1/p}|$ was recently
solved in  \cite{CDLY}. More specifically, if $L$ satisfies the   Gaussian estimate \eqref{GE}, then
 for every  $p\in (1, \infty)$ there exists a  constant $C=C(n,m, p)>0$ independent of $t$ such that
 \begin{eqnarray} \label{e1.5}
 \left\| (I+L)^{-s }e^{itL} f\right\|_{p} \leq C (1+|t|)^{s} \|f\|_{p}, \ \ \ t\in{\mathbb R}, \ \
  \ s= n\big|{1\over  2}-{1\over  p}\big|.
\end{eqnarray}
However, this result does not give any end-point estimate on the Hardy space $H^1_L(X)$ when $p=1.$

This paper   continues a line of study in \cite{CDLY} to  show that the operator
$(I+L)^{-n/2 }e^{itL}$ is bounded   on  Hardy spaces $H^1_L(X)$
under the assumption that
 $L$ satisfies   {\it $m$-th order Davies-Gaffney  estimates},
that is,  there exist constants $C, c>0$ such that
for all $t>0$,  and all $x,y\in X,$
 \begin{equation*}
 \label{DG}
 \tag{${\rm DG}_m$}
\big\|P_{B(x, t^{1/m})} e^{-tL} P_{B(y, t^{1/m})}\big\|_{2\to {2}}\leq
C   \exp\left(-c\left({d(x,y) \over    t^{1/m}}\right)^{m\over m-1}\right)
\end{equation*}
where $P_{B(x, t^{1/m})}$ denotes the characteristic function on the ball $B(x, t^{1/m})$
and $H^1_L(X)$  denotes the  Hardy space   associated with  $L$ (\cite{ADM, DY, KU},
see Section 2  below).
  We then apply the duality argument and the   complex interpolation result (see Lemma 4.1 below ) to obtain
 a new  proof of
 estimate \eqref{e1.5} in \cite{CDLY}    in the case that the operator $L$
satisfies  a  Gaussian upper bound \eqref{GE}.

Note that the $m$-th order Davies-Gaffney  estimate \eqref{DG} is much more general than the Gaussian estimate \eqref{GE}.
Indeed,  if an operator $L$ satisfies the \eqref{GE} estimate, then $L$ satisfies the \eqref{DG} estimate.
However, there are large classes of operators which satisfy  the \eqref{DG} estimate but not the
 \eqref{GE}  estimate.
  This happens, e.g., for Schr\"odinger operators with rough
 potentials \cite{ScV}, second order elliptic operators with rough  lower order terms \cite{LSV}, or
 higher order elliptic operators with bounded measurable coefficients
 \cite{D2}. See also \cite{Bl,  BK2, COSY, KU}.

 Our result can be stated as  follows.

\begin{thm}\label{th1.1}
Suppose  that $(X, d, \mu)$ is  a  space of homogeneous type  with a dimension $n$.  Suppose that $L$
satisfies the property \eqref{DG}.
Then  there exists a  constant $C=C(n, m)>0$ independent of $t$ such that
\begin{eqnarray} \label{e1.6}
  \left\| (I+L)^{-n/2 }e^{itL} f\right\|_{L^1(X)} +  \left\| (I+L)^{-n/2 }e^{itL} f\right\|_{H_L^1(X)}
  \leq C (1+|t|)^{n/2} \|f\|_{H_L^1(X)}, \ \ \ t\in{\mathbb R}.
\end{eqnarray}
By interpolation and duality argument, we have that for $1<p\leq 2$,
 \begin{eqnarray} \label{e1.555}
 \left\| (I+L)^{-s }e^{itL} f\right\|_{L^p(X)} \leq C (1+|t|)^{s} \|f\|_{H_L^p(X)}, \ \ \ t\in{\mathbb R}, \ \
  \ s= n\big|{1\over  2}-{1\over  p}\big| 
\end{eqnarray}
and for $2<p<\infty$,
 \begin{eqnarray} \label{e1.55522}
 \left\| (I+L)^{-s }e^{itL} f\right\|_{H^p_L(X)} \leq C (1+|t|)^{s} \|f\|_{L^p(X)}, \ \ \ t\in{\mathbb R}, \ \
  \ s= n\big|{1\over  2}-{1\over  p}\big|.
\end{eqnarray}

\end{thm}

\medskip

\begin{rem}

	(i) First, we would like to remark that our main result, Theorem~\ref{th1.1}, implies the main result in  \cite{CDLY} which states that
	under the generalised Gaussian estimate, we can obtain the sharp estimate for the Shr\"odinger group on Lebesgue spaces.
	For the convenience of the reader, we recall that the semigroup $e^{-tL}$    satisfies
	the generalized Gaussian  $(p_0, p'_0)$-estimate   of order $m$ (in which $1 \le p_0 < 2$ and ${1\over p_0} + {1\over p_0'} = 1)$,
	if there exist constants $C, c>0$ such that
	\begin{equation*}
	\label{GGE}
	\tag{${\rm GGE_{p_0,p'_0, m} }$}
	\big\|P_{B(x, t^{1/m})} e^{-tL} P_{B(y, t^{1/m})}\big\|_{p_0\to {p'_0}}\leq
	C V(x,t^{1/m})^{-({\frac{1}{ p_0}}-{1\over p'_0})} \exp\left(-c\left({d(x,y)^m \over    t }\right)^{1\over m-1}\right)
	\end{equation*}
	for  every $t>0$ and $x, y\in X$.
	In \cite{CDLY},  sharp estimate for the Shr\"odinger group on $L^p(X)$ was obtained for the range $p_0 < p < p_0'$
	under the assumption of $({\rm GGE_{p_0,p'_0, m} })$.
	
	Observe that by H\"older's inequality,  the generalized Gaussian  $(p_0, p'_0)$-estimate implies the generalized
	Gaussian  $(p_1, p'_1)$-estimate for $1 \le p_0 < p_1 \le 2$, hence implies  the  Davies-Gaffney  estimate \eqref{DG}
	(which is precisely the generalized Gaussian  $(p_0, p'_0)$-estimate when $p_0 =2$), see for example \cite{Bl}.
	Note also that under  the generalized Gaussian  $(p_0, p'_0)$-estimate, the Hardy space associated to operator
	$H^p_L(X)$ coincides with $L^p(X)$ for $p_0 < p \le 2$ (See for example \cite{KU}). Hence this paper gives a new proof for the
	main result in \cite{CDLY}.

	(ii) This paper gives a new end-point estimate on the Hardy space  $H^1_L(X)$ for large classes of operators which only require the
	$m$-th order Davies-Gaffney  estimate \eqref{DG}. Our result
	gives the sharp endpoint  estimate  \eqref{e1.6}
	for the  Schr\"odinger group $e^{itL}$ on the Hardy space, namely with the optimal number of derivatives
	and the optimal time growth for the factor
	$(1+|t|)^{s}$ in \eqref{e1.6}. While our endpoint estimate is obtained in terms of the Hardy space $H^1_L(X)$ associated to the operator $L$ instead of
	the classical Hardy space in the sense of Coifman and Weiss, it is known that if we assume stronger standard conditions on the operator $L$ such as
	the Gaussian estimate \eqref{GE} and H\"older continuity on the heat kernel, and
	the conservation property
	$
	e^{-tL}1 = 1,
	$
	then the Hardy space $H^1_L(X)$ associated to the operator $L$ coincides with
	the classical Hardy space.
	
	Note that when $L$ is the Laplace operator $\Delta$ on the Euclidean spaces ${\mathbb R}^n$,
  our Theorem~\ref{th1.1}  gives a direct proof of the following result:
  \begin{eqnarray}\label{zzz}
   \left\|  (1+\Delta)^{-n/2} e^{it\Delta}  f\right\|_{H^1(\mathbb R^n)}
  &\leq &   C(1+|t|)^{n/2} \| f\|_{H^1 (\mathbb R^n)}.
\end{eqnarray}
In \cite{Mi1},  Miyachi proved  the above estimate \eqref{zzz}    by using interpolation between $H^p(\mathbb R^n)$ for $p<1$ and $L^2(\mathbb R^n)$
for the Schr\"odinger group  $e^{it\Delta}$ on the Euclidean space $\mathbb R^n.$
	
	(iii) We also remark that the results
	in \cite{FS, Mi1, Mi} relies on Fourier analysis (e.g., Plancherel's Theorem),
	which is not available in the setting of space of homogeneous type in this paper.
	In the proof of Theorem~\ref{th1.1}, the main tool is to use the Phragm\'en-Lindel\"of theorem to show
	that the $m$-th order Davies-Gaffney  estimate \eqref{DG}   implies
	the following off-diagonal estimate of the operator
	$ e^{zL} $ with  $ z=(i\tau-1)R^{-1}, \tau, R>0$:
	\begin{eqnarray} \label{evv}
	\|P_{B}   e^{(i\tau-1)R^{-1}L} P_{2^jB\backslash 2^{j-1}B} \|_{2\to 2}\leq
	C
	\exp\left(-c\left({ \sqrt[m]{R} 2^j r \over    \sqrt{1+\tau^2}}\right)^{\frac{m}{m-1}}   \right),\ \ \ j=2,3,\ldots
	\end{eqnarray}
	for all  balls $B\subseteq  X$
	(see Lemma~\ref{le2.2} below).
	This new estimate  \eqref{evv} is   crucial in the proof of Theorem~\ref{th1.1}.
	
(iv)	In Section 5 we   apply Theorem~\ref{th1.1}
	to   the Schr\"odinger group
	of the Kohn Laplacian $\Box_b$  on polynomial model domains treated by Nagel-Stein~\cite{NS},
	where  $e^{-t\Box_b}$  satisfies $m$-th order Davies-Gaffney  estimates~\eqref{DG} with $m=2$.
	We note that in general polynomial model domains, $e^{-t\Box_b}$  does not satisfy  the generalized Gaussian  $(p_0, p'_0)$-estimate
	hence the result in \cite{CDLY} is not applicable to the Schr\"odinger group of the Kohn Laplacian $\Box_b$. The reason for $e^{-t\Box_b}$
	not satisfying the generalised Gaussian estimate is that $e^{-t\Box_b}$ could have singularity on the diagonal
	since the null space of $\Box_b$ may not be $\{0\}$.  It is worth pointing out that if the null space of $\Box_b$ is $\{0\}$,
	then $e^{-t\Box_b}$ satisfies the standard Gaussian upper bound, see for example \cite{BR}.
	
\end{rem}

The paper is organized as follows.
In Section 2 we provide some preliminary results on Hardy spaces
and   spectral multipliers. In Section 3 we   apply the Phragm\'en-Lindel\"of theorem
to
give off-diagonal bounds for \eqref{evv} and the operator $F(L)$
for some  compactly supported function $F$. This plays a crucial role in the proof of   Theorem~\ref{th1.1}
which  will be given in Section 4. In Section 5 we   give an application of  Theorem~\ref{th1.1}
in a study of   the Schr\"odinger group
for the Kohn Laplacian  on polynomial model domains.

\medskip

\section{Notations and preliminaries on Hardy spaces}
\setcounter{equation}{0}

We start by introducing  some notation and assumptions.  Throughout this paper,
unless we mention the contrary, $(X,d,\mu)$ is a metric measure  space where $\mu$
is a Borel measure with respect to the topology defined by the metric $d$.
Next, let
$B(x,r)=\{y\in X,\, {d}(x,y)< r\}$ be  the open ball
with centre $x\in X$ and radius $r>0$. To simplify notation we often just use $B$ instead of $B(x, r)$ and
given $\lambda>0$, we write $\lambda B$ for the $\lambda$-dilated ball
which is the ball with the same centre as $B$ and radius $\lambda r$. Let $B^c$ be the set $X\backslash B$.
We set $V(x,r)=\mu(B(x,r))$ the volume of $B(x,r)$ and we say that $(X, d, \mu)$ satisfies
 the doubling property (see Chapter 3, \cite{CW})
if there  exists a constant $C>0$ such that
\begin{eqnarray}
V(x,2r)\leq C V(x, r)\quad \forall\,r>0,\,x\in X. \label{eq2.1}
\end{eqnarray}
If this is the case, there exist  $C, n$ such that for all $\lambda\geq 1$ and $x\in X$
\begin{equation}
V(x, \lambda r)\leq C\lambda^n V(x,r). \label{eq2.2}
\end{equation}
In the Euclidean space with Lebesgue measure, $n$ corresponds to
the dimension of the space.

  For $1\le p\le+\infty$, we denote the
norm of a function $f\in L^p(X,{\rm d}\mu)$ by $\|f\|_p$, by $\langle \cdot,\cdot \rangle$
the scalar product of $L^2(X, {\rm d}\mu)$, and if $T$ is a bounded linear operator from $
L^p(X, {\rm d}\mu)$ to $L^q(X, {\rm d}\mu)$, $1\le p, \, q\le+\infty$, we write $\|T\|_{p\to q} $ for
the  operator norm of $T$.
Given a  subset $E\subseteq X$, we  denote by  $\chi_E$   the characteristic
function of   $E$ and  by $P_E$ the projection
$
P_Ef(x):=\chi_E(x) f(x).
$
We denote the dilation of a function $F$ by $\delta_r F(\cdot):=F(r\cdot)$ and
$\widehat{f}\,$ denotes the Fourier  transform, i.e.
of $f$,
$$
\widehat{f}(\xi)={1\over (2\pi)^{n/2}}\int_{\mathbb R^n} f(x)e^{-ix\xi} dx, \ \ \ \ \xi\in \mathbb R^n.
$$
Sometimes we also use  $\widehat{f}$ for ${\mathcal F} f$.

\medskip
\subsection{Hardy   spaces associated with operators}
 A theory of Hardy spaces associated with certain operators was introduced and developed in \cite{ADM, DL, DY2, DY3, HLMMY, KU}
 and the references therein, similar to
 the way that classical Hardy spaces are adapted to the Laplacian. We present some main features of this theory
 in this section for reader's convenience.

  Suppose that $L$ is a non-negative
 self-adjoint operator on $L^2(X)$ which satisfies $m$-th order Davies-Gaffney  estimates~\eqref{DG} with $m\geq 2$.
Following \cite{HLMMY}, we define the $L^2$ adapted Hardy space
$$
H^2(X):=H^2_L(X):=\overline{R(L)},
$$
that is, the closure of the range of $L$ in $L^2(X)$.
Then $L^2(X)$
is the orthogonal sum of $H^2(X)$ and the null space $N(L)$.
Consider the following quadratic operators associated 
to $L$
   
\begin{eqnarray}
S_{h, K}f(x)=\Big(\int_0^{\infty}\!\!\!\!\int_{\substack{  d(x,y)<t}}  
|(t^2L)^{K}e^{-t^2L} f(y)|^2 {d\mu(y)\over V(x,t)}{dt\over t}\Big)^{1/2},
\quad x\in X 
\label{e2.1}
\end{eqnarray}

\noindent
where $f\in L^2(X)$. We shall write $S_{h}$ in place of $S_{h, 1}$. 
For each $K\geq 1$ and $1\leq p<\infty$, we now define
  
\begin{eqnarray*} 
D_{K, p} =\Big\{ f\in H^2(X): \ S_{h, K}f\in L^p(X)\Big\}, \ \ \ 
1\leq p<\infty.
\end{eqnarray*}

\begin{definition}{\label{def2.2}} Let $L$ be  
 a  self-adjoint positive definite  operator on $L^2({X})$ 
 satisfying the Davies-Gaffney estimate (\ref{e1.5}). 
 
 \medskip
 
(i) For each $1\leq p\leq 2$, the Hardy space $H^p_{L, S_h}(X)$ associated to
$L$  is the completion of the space $D_{1, p}$ in the norm
$$
 \|f\|_{H_{L, S_h}^p(X)}=  \|S_{h}f\|_{L^p(X)}.
$$

(ii) For each $2<p<\infty$, the Hardy space $H^p_{L}(X)$ associated to
$L$ is the completion of the space $D_{K_0, p}$ in the norm 
$$
\|f\|_{H_{L, S_h}^p(X)}=  \|S_{h, K_0}f\|_{L^p(X)}, \ \ \ \ 
K_0=\left[\,{n\over 4}\,\right]+1.
$$
\end{definition}

The Hardy spaces associated to $L$ are known to possess nice properties,
for example,   they form a complex interpolation scale (see Lemma~\ref{th7.55} below).
Note that, in the framework of 
the present paper, we only assume the Davies-Gaffney estimates on the heat 
kernel of $L$, and hence for $1<p<\infty$, $p\not= 2$, 
$H^p_{L, S_h}(X)$ may or 
may not coincide with the space $L^p(X)$.  However, it can be verified that 
$H^2_{L, S_h}(X)=H^2(X)$.
 It remains an open problem, in this general context, to determine whether
$H^1_L(X)\subseteq L^1(X)$  (see \cite[p. 70]{HLMMY} and \cite{AMM}).

 Let us describe the notion of a
{\it $(1,2,M,\varepsilon)$-molecule}
associated to an operator $L$  on spaces $(X,d,\mu)$.  Denote by ${\mathcal D}(T)$
the domain of an operator $T$. For every ball $B$, we set
\begin{equation}
U_0(B)=B,  \ \  {\rm and}\ \ \ U_j(B)=2^j B\backslash 2^{j-1}B
\,\,\mbox{ for }\,\,j=1,2, \dots.
\label{e2.3}
\end{equation}

\begin{definition} \label{def2.7} Let   $ \varepsilon>0$ and
$M\in{\mathbb N}$.
A function $m(x)\in L^2(X)$ is called a
$(1,2,M,\varepsilon)$-molecule associated with $L$ if there exist a function
$b\in {\mathcal D}(L^M)$ and a ball $B$ such that

\medskip

{ (i)}\ $m=L^M b$;

\medskip

{ (ii)}\  For every $k=0,1,2,\dots,M$ and $j=0,1,2,\dots$, there holds
$$
\|(r_B^mL)^{k}b\|_{L^2(U_j(B))}\leq 2^{-j\varepsilon} r_B^{mM}
V(2^jB)^{-{1/2}},
$$
\noindent where the annuli $U_j(B)$ are defined in (\ref{e2.3}).
\end{definition}

Next we give the definition of the molecular Hardy spaces associated with $L$.

\begin{definition}\label{def 2.4} We fix $\varepsilon > 0$ and
$M\in{\mathbb N}$.
The Hardy space $\HML$ is defined as follows.
We say that
$f= \sum\lambda_j m_j$ is a molecular
$(1,2,M,\varepsilon)$-representation (of $f$) if $ \{\lambda_j\}_{j=0}^{\infty}\in {\ell}^1$,
each $m_j$ is a $(1,2,M,\varepsilon)$-molecule, and the sum converges in $L^2(X).$
Set
\begin{equation*}
\HHML= \Big\{f:
\mbox{$f$ has a molecular $(1,2,M,\varepsilon)$-representation} \Big\},
\end{equation*}
with the norm given by
$$
||f||_{\HML}= {\rm inf}\Big\{\sum_{j=0}^{\infty}|\lambda_j|: f=\sum\limits_{j=0}^{\infty}\lambda_jm_j\,
\mbox{ is a molecular $(1,2,M,\varepsilon)$-representation} \Big\}.
$$
The space $\HML$ is then defined as the completion of $\HHML$ with respect to this norm.
\end{definition}

As a direct consequence of the definition, we note that $H^1_{L, {mol}, M_2, \varepsilon}(X)\subset
H^1_{L, {mol}, M_1, \varepsilon}(X)$ for $\varepsilon>0$ and $M_1, M_2\in{\mathbb N}$ with $M_1\leq M_2$.
We  have the following characterization. For its  proof, see\cite[ Section 3]{DL}.

\begin{lemma}\label{prop2.10} Suppose   $M\geq n/4$.
  Then we have $\HML = \HSL$.  Moreover,
$$
\|f\|_{\HML} \approx \|f\|_{H^1_{L}(X)},
$$
\noindent
where the implicit constants depend only on $M, m$ and  $n$   in \eqref{eq2.2}  only.
\end{lemma}

We have the following dual result.
\begin{lemma}\label{th77.55}
Assume that the operator $L$ satisfies $m$-th order Davies-Gaffney  estimates~\eqref{DG} with $m\geq 2$.   Then  for $1<p<\infty$,  we have
$$
(H^p_{L}(X))^{\ast}=H^{p'}_{L}(X),
$$
where $p'$ is the conjugate index of $p$ such that $1/p'+1/p=1$.
\end{lemma}

 Similar to the classical Hardy spaces, Hardy spaces associated with operators
form a complex interpolation scale.
Let $[\cdot, \, \cdot]_{\theta}$ stand  for the complex interpolation bracket. Then we have the following result.

\begin{lemma}\label{th7.55}
Assume that the operator $L$ satisfies $m$-th order Davies-Gaffney  estimates~\eqref{DG} with $m\geq 2$.
 Then  for every  $0<\theta<1$ and $1<p_0<\infty$,  we have
$$
[ H^1_L,   H_L^{p_0} ]_{\theta} =H_L^{p}, \ \ \ \  \  {1\over p}=(1-\theta) +{\theta\over p_0},
$$
\end{lemma}

\begin{proof}
The proof can be verified that by viewing these spaces via the framework of tent spaces
and by using the interpolation properties  of tent spaces (see for example,  \cite[Lemma 4.20]{KU}).
\end{proof}

 \medskip
\subsection{Spectral multipliers  on the Hardy space.}
 The following result is a standard known result in the theory of spectral multipliers of non-negative self-adjoint
operators.

\begin{prop}\label{prop2.7}
Let $m\geq 2$. Suppose that $(X, d, \mu)$ is a space of homogeneous type with a dimension $n$.
Assume that the operator $L$ satisfies the $m$-th order Davies-Gaffney  estimates~\eqref{DG} with $m\geq 2$.
Assume in addition that $F$ is  an even bounded Borel function  such that
 $
\sup_{R>0}\|\eta\delta_RF\|_{C^\alpha}<\infty
 $
for some integer $\alpha> (n+1)/2  $ and some non-trivial function $\eta\in C_c^{\infty}(0, \infty)$.
Then the operator  $F(L)$ is bounded on $H^{1}_L(X)$,
\begin{eqnarray}\label{n}
\|F({L})\|_{H^1_L(X)\to  H^1_L(X)}\leq C \left(\sup_{R>0}\|\eta\delta_RF\|_{C^\alpha}+F(0)\right).
\end{eqnarray}
\end{prop}

\begin{proof}
For the proof,  see for example, \cite[Theorem 1.4]{KU} and \cite[Theorem 1.1]{DY3}.
\end{proof}

\bigskip

\section{Off-diagonal bounds for compactly  supported spectral multipliers  }
\setcounter{equation}{0}

Let us start with stating the Phragm\'en-Lindel\"of Theorem
 for sectors in the complex plane $\mathbb C$. For its proof, we refer to \cite[Lemma 4.2]{SW}.

\begin{thm}\label{th2.1} Let $S$ be the open region in ${\mathbb C}$ bounded by two rays meeting at an angle $\pi/a$ for some $a>1/2$.
Suppose that $F$ is analytic on $S$, continuous on ${\bar S}$ and satisfies $|F(z)|\leq C\exp (c|z|^{b})$ for some $b\in [0, a)$
and for all $z\in S$. Then the condition $|F(z)|\leq B$
 on the two bounding rays implies that $|F(z)|\leq B$ for all $z\in S$.
 \end{thm}

  The following result is a   consequence of
 Theorem~\ref{th2.1}.

 \begin{lemma}\label{le2.22}
\label{tw0} Suppose that $F$ is an analytic function on $\mathbb{C}_+=\{z\in{\mathbb C}:  {{\rm Re }z}>0 \}$, the open
right half-plane.
Assume that, for given numbers $M_1, M_2, \gamma>0$, $0<\alpha\leq 1$,
\begin{equation}  \label{a1}
|F(z)|\le M_1, \quad \forall \, z \in \mathbb{C}_+
\end{equation}
and
\begin{equation}  \label{a2}
|F(t)|\le M_2 \exp\Big({-\frac{\gamma}{t^\alpha}}\Big), \quad \forall \, t \in \mathbb{R}_+.
\end{equation}
Then for every $ z \in \mathbb{C}_+$,
\begin{equation}  \label{bb2}
|F(z)| \le \max\big\{ M_1 ,M_2 \big\}\exp \left(-\alpha \gamma \, { {{\rm Re }z} \over |z|^{\alpha+1} }\right).
\end{equation}
\end{lemma}

\begin{proof}  Lemma~\ref{le2.22} was proved in  \cite[Lemma 9]{D2}. See also  \cite[Proposition 2.2]{CouS}
 and \cite[Lemma 6.18]{O}.
   We give a brief argument
of this proof for completeness and convenience for the reader.

Consider the function
\begin{equation}  \label{gg1}
u_+(\zeta):=F\left(\frac{1}{\zeta}\right)\exp\left(\gamma e^{i(\pi/2-\pi \alpha/2)}{\zeta^\alpha}\right),
\end{equation}
which is also defined on $\mathbb{C}_+$.
By (\ref{a1}),
\begin{equation*}
|u_+(\zeta)|\le M_1 \exp\left(\gamma|\zeta|^\alpha\right), \quad \forall \, \zeta \in \mathbb{C}_+ .
\end{equation*}
Again by (\ref{a1}) we have, for any $\varepsilon > 0$ and $\zeta=\varepsilon+iy=:C_\varepsilon e^{i\theta_\varepsilon}$,
\begin{eqnarray*}
|u_+(\zeta)|&=&\left|F\left(\frac{1}{\zeta}\right)\exp\left(\gamma e^{i(\pi/2-\pi \alpha/2)}{\zeta^\alpha}\right)\right|
 \leq  M_1\exp\left( C_\varepsilon^\alpha \gamma\sin\Big(\frac{\pi \alpha}{2}-\alpha\theta_\varepsilon\Big) \right).
\end{eqnarray*}
For $y\geq 0$, it follows from $0<\alpha\leq 1$ that
\begin{eqnarray*}
|u_+(\zeta)|
&\leq& M_1\exp\left( C_\varepsilon^\alpha \gamma\sin\Big(\frac{\pi \alpha}{2}-\alpha\theta_\varepsilon\Big) \right)
 \leq  M_1\exp\left(  \gamma\varepsilon^\alpha \right),
\end{eqnarray*}
which implies that
\begin{equation}  \label{22a}
\qquad \sup_{{\mbox{\Small{\rm Re}}}\zeta=\varepsilon, {\mbox\Small{\rm Im}}\zeta\geq 0}|u_+(\zeta)|\le
M_1 e^{\gamma \varepsilon^\alpha}.
\end{equation}
By (\ref{a2}),
\begin{equation}  \label{22b}
\sup_{ \zeta\in [\varepsilon,\infty)}|u_+(\zeta)|\le M_2.
\end{equation}
Hence, by Phragm\'en-Lindel\"of theorem~\ref{th2.1} with angle $\pi/2$ and $b=\alpha$,
applied to
\begin{equation*}
S^+_{\varepsilon}= \{z\in \mathbb{C} \colon \, {\mbox{\Small{\rm Re}}} z >
\varepsilon \quad \mbox{and} \quad {\mbox{\small{\rm Im}}} z > 0\},
\end{equation*}
we  obtain
\begin{equation*}
\sup_{{\mbox{\Small{\rm Re}}}\zeta\geq\varepsilon, {\mbox\Small{\rm Im}}\zeta\geq 0}|u_+(\zeta)|\le
\max\{ M_2 , M_1 e^{\gamma \varepsilon^a} \}, \quad \forall \,
\varepsilon > 0 .
\end{equation*}

Next  we consider the function
\begin{equation}  \label{ggg2}
u_-(\zeta):=F\left(\frac{1}{\zeta}\right)\exp\left(\gamma e^{i(-\pi/2+\pi \alpha/2)}{\zeta^\alpha}\right).
\end{equation}
A similar argument  shows that
\begin{equation*}
\sup_{{\mbox{\Small{\rm Re}}}\zeta\geq\varepsilon,\,  {\mbox\Small{\rm Im}}\zeta\leq 0}|u_-(\zeta)|\le
\max\{M_2 , M_1e^{\gamma \varepsilon^a} \}, \quad \forall \,
\varepsilon > 0 .
\end{equation*}
Letting $\varepsilon \to 0$ we obtain
\begin{equation*}
\sup_{{\mbox{\Small{\rm Re}}}\zeta>0, \, {\mbox\Small{\rm Im}}\zeta\geq 0}|u_+(\zeta)|\le
\max\{ M_1 ,M_2 \}
\end{equation*}
and
\begin{equation*}
\sup_{{\mbox{\Small{\rm Re}}}\zeta>0, \,  {\mbox\Small{\rm Im}}\zeta\leq 0}|u_-(\zeta)|\le
\max\{ M_1, M_2\}.
\end{equation*}
Putting $\zeta =\frac{1}{z}$, we obtain for all $\mbox{Re}\,z>0$
\begin{equation*}
|F(z)|\le
\max\{ M_1, M_2\}\exp\left(- \alpha\sin\left(\pi/2+|\theta_z|\right) {\gamma \over |z|^{\alpha}}\right),
\end{equation*}
where $\theta_z=\arg z$. From this, (\ref{bb2}) follows readily.
 \end{proof}

\begin{lemma}\label{le2.2}
Suppose that  $L$ satisfies the $m$-th order Davies-Gaffney  estimates~\eqref{DG} with $m\geq 2$.
There exist  two positive constants  $C$ and $c$ such that
for every $j=2,3,\ldots$
 \begin{eqnarray*}
\|P_{B}   e^{(i\tau-1)R^{-1}L} P_{U_j(B)} \|_{2\to 2}&\leq&
C
  \exp\left(-c\left({ \sqrt[m]{R} 2^j r \over    \sqrt{1+\tau^2}}\right)^{\frac{m}{m-1}}   \right)
\end{eqnarray*}
for all  balls $B\subseteq  X$.
 \end{lemma}

\begin{proof}   For any open sets $U$ and $V$, and ${\rm Re}z>0$,
 we define a function
 $$
 F(z):=\langle e^{-zL}f_1,f_2\rangle,
 $$
 where $\supp f_1\subset U$ and $\supp f_2\subset V$.
 Then $F(z)$ is an analytic function  on the complex half plain ${\rm Re} \, z>0$.  It is seen that
 $$
 |F(z)|\leq \|e^{-zL}f_1\|_{2} \|f_2\|_2\leq \|e^{-z\lambda}\|_{L^\infty}\|f_1\|_2\|f_2\|_2\leq
  \|f_1\|_2\|f_2\|_2, \quad \forall \, z \in \mathbb{C}_+
 $$
 and it follows from the $m$-th order Davies-Gaffney  estimates~\eqref{DG}  and \cite[Theorem 1.2]{BK} that
 $$
 |F(t)|\le C \exp\left(-c{ \frac{d(U,V)^{\frac{m}{m-1}}}{ t^{\frac{1}{m-1}}}}\right)\|f_1\|_2\|f_2\|_2, \quad \forall \, t \in \mathbb{R}_+.
 $$
  Let $M_1=\|f_1\|_2\|f_2\|_2$, $M_2=C\|f_1\|_2\|f_2\|_2$,
 $\gamma= cd(U,V)^{m/(m-1)} $ and $\alpha=1/(m-1)$.
We apply  Lemma~\ref{le2.22} to get
 $$
 |\langle e^{-zL}f_1,f_2\rangle|\leq C\exp\left( {-  c\, {{\rm Re} z}\, \frac{d(U,V)^{\frac{m}{m-1}} } { |z|^{\frac{1}{m-1}+1}} }\right)\|f_1\|_{2} \|f_2\|_2.
 $$
From it, we have that
 \begin{eqnarray*}
\|P_{B}   e^{(i\tau-1)R^{-1}L} P_{U_j(B)} \|_{2\to 2}&\leq&
C
  \exp\left(-c \left({ \sqrt[m]{R} 2^j r_B \over    \sqrt{1+\tau^2}}\right)^{\frac{m}{m-1}}   \right).
\end{eqnarray*}
This ends the proof of Lemma~\ref{le2.2}.
\end{proof}

Next we define a Besov type norm of $F$ by
$$
\|F\|_{{{\bf B}^{s}}}:=\int_{-\infty}^{\infty} |\widehat{F}(\tau)|(1+|\tau|)^{s}d\tau,
$$
 where ${\widehat f}$ denotes the Fourier transform of $f$.
Since for every  functions $F$ and $G$,  it can be checked that
 \begin{eqnarray*}
 \|FG\|_{{{\bf B}^{s}}}&=&\int_{-\infty}^{\infty} |\widehat{(FG)}(\tau)|(1+|\tau|)^{s}d\tau\\
&\leq & \int_{-\infty}^{\infty} \int_{-\infty}^{\infty}\big|  (\widehat{F} (\tau-\eta)   \widehat{G}(\eta) \big|
 |(1+|\tau-\eta|)^{s} |(1+|\eta|)^{s}d\eta  d\tau\\
\end{eqnarray*}
 and so by the Fubini theorem,
$$
\|FG\|_{{{\bf B}^{s}}}\leq \|F\|_{{{\bf B}^{s}}}\|G\|_{{{\bf B}^{s}}}.
$$

Finally, we can show the following result.

\begin{prop}\label{le2.3}
Suppose that  $L$ satisfies  the  Gaussian upper bounds \eqref{DG} with $m\geq 2$.
Then for every  $s\geq 0$, there exists  a constant $C>0$ such that
for every $j=2,3,\ldots$
\begin{eqnarray}\label{e3.8} \hspace{1cm}
 \big\|P_{B}    F({L}) P_{U_j(B)}\big\|_{2\to 2}\leq
 C  \big(\sqrt[m]{R}2^j r)^{-s}   \|F(R\cdot)\|_{{{\bf B}^{s}}}
\end{eqnarray}
for all  balls $B\subseteq  X$, and all Borel functions $F$  such that supp $F\subseteq [-R, R]$.
 \end{prop}

\begin{proof}  Let $G(\lambda)=F({R\lambda}) e^{\lambda}.$ In virtue of the Fourier inversion formula
$$
F(L)=G(L/R)e^{-L/R}={1\over 2\pi} \int_{\mathbb R} e^{(i\tau-1)R^{-1}L} {\hat G}(\tau)d\tau
$$
we have that
$$
 \|P_{B}    F({L}) P_{U_j(B)}\|_{2\to 2} \leq
{1\over 2\pi} \int_{\R} |{\hat G}  (\tau)| \,   \|P_{B}   e^{(i\tau-1)R^{-1}L} P_{U_j(B)} \|_{2\to 2}
  d\tau.
$$
   Then it follows from Lemma~\ref{le2.2} for every $s\geq 0$,
 \begin{eqnarray*}
\|P_{B}   e^{(i\tau-1)R^{-1}L} P_{U_j(B)} \|_{2\to 2}&\leq&
C
  \exp\left(-c\left({ \sqrt[m]{R} 2^j r \over    \sqrt{1+\tau^2}}\right)^{\frac{m}{m-1}}   \right)\\
  &\leq& C_s     \left( {  \sqrt[m]{R} 2^jr \over \sqrt{1+\tau^2}    } \right)^{-s}.
\end{eqnarray*}
Therefore (compare \cite[(4.4)]{DOS})
  \begin{eqnarray*}
   \|P_{B}    F({L}) P_{U_j(B)}\|_{2\to 2}
   &\leq&
 C \big(\sqrt[m]{R} 2^j r)^{-s}
   \int_{\R} |{\hat G}  (\tau)|  \big(1+|\tau|)^{ s}  d\tau\nonumber\\
   &\leq&    C   \big(\sqrt[m]{R} 2^j r)^{-s}   \|G\|_{{{\bf B}^{s}}}.
\end{eqnarray*}
 Note that supp $F\subseteq [-R, R]$ and so supp $ F(R\cdot)\subseteq  [-1, 1]$. Thus taking a function
 $\psi\in C_c^\infty$ such that supp $\psi\subset [-2,2]$ and $\psi(\lambda)=1$ for $\lambda\in [-1,1]$, we have
 $$
 G(\lambda)=F({R\lambda}) e^{\lambda}=F({R\lambda}) \psi(\lambda)e^{\lambda}
 $$
 and so
 $$
  \|G\|_{{{\bf B}^{s}}} \leq C \|F(R\cdot)\|_{{{\bf B}^{s}}} \|\psi(\lambda)e^{\lambda}\|_{{{\bf B}^{s}}}
  \leq C \|F(R\cdot)\|_{{{\bf B}^{s}}} .
 $$
This ends the proof of Proposition~\ref{le2.3}.
\end{proof}

\begin{rem}\label{re3.4}\
In \cite[Proposition 4.1]{CCO}, Carron, Coulhon and Ouhabaz used some techniques introduced by Davies (\cite{D2})  to show that the upper
Gaussian estimate \eqref{GE}  on $e^{-tL}, t>0, $ extends to a similar estimate on $e^{-zL}$ where
$z$ belongs to  the whole complex right half-plane  and all $x, y\in X$,
\begin{eqnarray*}
|p_z(x,y)|\leq {C\over \left(V(x, ({|z|\over (\cos \theta)^{m-1} })^{1/m} )\, V(y, ({|z|\over (\cos \theta)^{m-1} })^{1/m} )\right)^{1/2}}
 \exp\left(-c \, {  \left({d(x,y)^{m}\over    |z|}\right)^{1\over m-1}}\cos\theta \right) {1\over (\cos\theta)^{n}}
 \end{eqnarray*}
where $\theta={\rm Arg} \, z$. It follows that
  for every $j=2,3,\ldots$
 \begin{eqnarray} \label{ecc}
\|P_{B}   e^{(i\tau-1)R^{-1}L} P_{U_j(B)} \|_{2\to 2}&\leq&
C 2^{jn}
  \exp\left(-c\left({ \sqrt[m]{R} 2^j r \over    \sqrt{1+\tau^2}}\right)^{\frac{m}{m-1}}   \right)
\end{eqnarray}
for all  balls $B=B(x_B, r)\subseteq  X$. In our Lemma~\ref{le2.2}, we made an important improvement in obtaining
the upper bound on the right hand side  of \eqref{ecc} without the factor ``$2^{jn}$".
This plays an essential role in
 estimate \eqref{e3.8} of Proposition~\ref{le2.3} and
in the proof of Theorem~\ref{th1.1} in Section 4.
\end{rem}

\bigskip

\section{Proof of Theorem~\ref{th1.1}}
\setcounter{equation}{0}

To prove \eqref{e1.6},  let us show that
\begin{eqnarray} \label{e1.66}
 \left\| (I+L)^{-n/2 }e^{itL} f\right\|_{H_L^1(X)} \leq C (1+|t|)^{n/2} \|f\|_{H_L^1(X)}, \ \   t\in{\mathbb R}.
\end{eqnarray}
The proof of  estimate of $\| (I+L)^{-n/2 }e^{itL} f\|_{L^1(X)} $ uses similar ideas,  but it is much  simpler.
In the following,
$\phi$ denotes  a non-negative $C_c^{\infty}$ function on $\mathbb R$ such
that ${\rm supp}\  \phi \subseteq ({1/4}, 1)$
and let
$\phi_\tau(\lambda)$ denote the function $\phi(\tau\lambda)$. Also we let
 $\psi\in C_c^\infty$   supported in  $\psi\subset [1/8, 2]$ and $\psi(\lambda)=1$ for $\lambda\in [1/4,1]$.
To prove   \eqref{e1.66}, it follows by Lemma~\ref{prop2.10}  and a standard argument (see for example, \cite{DY3, HLMMY, HM, KU}) that
it suffices to show  that for every $(1,2,M,\varepsilon)$-molecule $a$ associated to a
ball $B$,
\begin{eqnarray}\label{e3.3} \hspace{0.5cm}
\left\|\left(\int_0^{\infty}\!\!\!\!\int_{\substack{  d(x,y)<\tau^{1/m}}}
|\phi(\tau L)(1+L)^{-n/2}e^{itL}a(y)|^2 {d\mu(y)\over V(x,\tau^{1/m})}{d\tau\over \tau}\right)^{1/2}\right\|_{L^1}\leq C(1+|t|)^{n/2}, \ \ \  t\in{\mathbb R}
\end{eqnarray}
where $M\in{\mathbb N}$  is large enough so that $M>n/2.$

Recall that if $a$ is a $(1,2,M,\varepsilon)$-molecule
associated to a ball $B=B(x_B,r)$, then there exists a function $b$ such that $a=L^Mb$ and for every $k=0,1,2,\dots,M$ and $j=0,1,2,\dots$, there holds
\begin{eqnarray}\label{ee3.0}
\|(r^mL)^{k}b\|_{L^2(U_j(B))}\leq 2^{-j\varepsilon} r^{mM}
V(2^jB)^{-{1/2}},
\end{eqnarray}
where the annuli $U_j(B)$ were defined in (\ref{e2.3}).
  Following  \cite{HM},
we write
\begin{eqnarray}
&&\hspace{-1.2cm}I=m\, \Big( r^{-m}\int_{r}^{\sqrt[m]{2}r}s^{m-1}ds\Big)   \cdot I\nonumber\\
&=&mr^{-m} \int_{r}^{\sqrt[m]{2}r}s^{m-1}(I-e^{-s^mL})^Mds
 +  \sum_{\nu=1}^M C_{\nu,M}  r^{-m}
\int_{r}^{\sqrt[m]{2}r}
s^{m-1} e^{- \nu s^mL} ds,
\label{e3.5}
\end{eqnarray}
\noindent
where $C_{\nu, M}$ are some constants depending   on $\nu$ and $M$ only. However,
 $\partial_se^{-\nu s^mL}=-m\nu s^{m-1}Le^{-\nu s^mL} $ and therefore,
\begin{eqnarray}
m\nu L\int_{r}^{\sqrt[m]{2}r} s^{m-1} e^{-\nu s^mL} ds
&=&e^{-\nu r^mL}-e^{-2\nu r^mL}
=
e^{-\nu r^mL} (I-e^{-\nu r^mL})     \nonumber\\ [4pt]
&=&  e^{-\nu r^mL}
(I-e^{-r^mL})    \sum_{\mu=0}^{\nu -1}e^{-ir^mL}.
\label{e3.6}
\end{eqnarray}

In the following, we set $F_\tau(\lambda):=\phi(\tau \lambda)(1+\lambda)^{-n/2}e^{it\lambda}, t>0$.
Applying the procedure outline in (\ref{e3.5})-(\ref{e3.6})  $M$ times,
we have for every $x\in X,$
\begin{eqnarray*}
 F_\tau(L)a(x) &=&(1+L)^{-n/2}e^{itL}a(x)\\
&=&m^M\left(r^{-m}\int_r^{\sqrt[m] 2 r} s^{m-1}(1-e^{-s^mL})^M ds\right)^M F_\tau(L)a(x) \\
 &+&\sum_{k=1}^M (1-e^{-r^mL})^k \left(r^{-m}\int_r^{\sqrt[m] 2 r} s^{m-1}(1-e^{-s^mL})^{M}ds\right)^{M-k}\times\\
 &&\ \ \ \times  \sum_{\nu =1}^{(2M-1)k} C(\nu ,k, M) e^{-\nu  r^mL}F_\tau(L) (r^{-mk}L^{M-k} b)(x)\\
&=&\sum_{k=0}^{M-1} r^{-m}  \int_r^{\sqrt[m] 2 r} s^{m-1}(1-e^{-s^mL})^{M}G_{k,r,M}(L)F_\tau(L) (r^{-mk}L^{M-k} b)ds\\
 & +& \sum_{\nu =1}^{(2M-1)M} C(\nu, k, M) e^{-\nu  r^mL}F_\tau(L)(1-e^{-r^mL})^M(r^{-mM} b)(x)\\
&=:& \sum_{k=0}^{M-1} E_k(x)+ E_M(x),
\end{eqnarray*}
where for $k=1,2,\cdots,M-1$
$$
G_{k,r,M}(\lambda):=(1-e^{-r^m\lambda})^k \left(r^{-m}\int_r^{\sqrt[m] 2 r} s^{m-1}(1-e^{-s^m\lambda})^{M}ds\right)^{M-k-1}
\sum_{\nu =1}^{(2M-1)k} C(\nu,k, M) e^{-\nu r^m\lambda}
$$
and for $k=0$
$$
G_{0,r,M}(\lambda):=m^M\left(r^{-m}\int_r^{\sqrt[m] 2 r} s^{m-1}(1-e^{-s^m\lambda})^{M}ds\right)^{M-1}.
$$
We will establish an adequate bound on each $E_k, k=0, 1, \cdots, M$, by considering  two cases $k=0, 1, \cdots, M-1$
and $k=M.$

\smallskip
Next define
\begin{eqnarray} \label{eddde1}
F_{\tau,s}(\lambda):=\phi_\tau(\lambda)(1-e^{-s^m\lambda})^{M}G_{k,r,M}(\lambda)F(\lambda).
\end{eqnarray}
\noindent
{\bf Case 1.} \ $k=0, 1, \cdots, M-1$. \ In this case, we see that
  \begin{eqnarray}\label{e1}
&&\left\|\left(\int_0^{\infty}\!\!\!\!\int_{\substack{  d(x,y)<\tau^{1/m}}}
|E_k(y)|^2 {d\mu(y)\over V(x,\tau^{1/m})}{d\tau\over \tau}\right)^{1/2}\right\|_{L^1}\nonumber\\
&\leq& C\sup_{s\in [r,\sqrt[m] 2 r]}\left\|\left(\int_0^{\infty}\!\!\!\!\int_{\substack{  d(x,y)<\tau^{1/m}}}
|F_{\tau,s}(L) (r^{-mk}L^{M-k} b)(y)|^2 {d\mu(y)\over V(x,\tau^{1/m})}{d\tau\over \tau}\right)^{1/2}\right\|_{L^1}\nonumber\\
&\leq& C\sum_{j\geq 0}\sup_{s\in [r,\sqrt[m] 2 r]}\left\|\left(\int_0^{\infty}\!\!\!\!\int_{\substack{  d(x,y)<\tau^{1/m}}}
|F_{\tau,s}(L)P_{U_{j}(B)} (r^{-mk}L^{M-k} b)(y)|^2 {d\mu(y)\over V(x,\tau^{1/m})}{d\tau\over \tau}\right)^{1/2}\right\|_{L^1}\nonumber\\
&=:& C\sum_{j\geq 0}\sup_{s\in [r,\sqrt[m] 2 r]} \|E(k, j, s)\|_{L^1(X)},
\end{eqnarray}
where
$$
E(k, j, s)=\left(\int_0^{\infty}\!\!\!\!\int_{\substack{  d(x,y)<\tau^{1/m}}}
|F_{\tau,s}(L)P_{U_{j}(B)} (r^{-mk}L^{M-k} b)(y)|^2 {d\mu(y)\over V(x,\tau^{1/m})}{d\tau\over \tau}\right)^{1/2}.
$$

Let us estimate the term $\|E(k, j, s)\|_{L^1(X)}.$ Note that  $\|G_{k,r,M}\|_{L^\infty}+ \|F\|_{L^\infty}\leq C$.
We apply estimate  \eqref{ee3.0}, the $L^2$-boundedness of the area square function  and   the doubling condition \eqref{eq2.2},
$$
\left({\mu((1+t)2^j B) \over  \mu(2^j B)} \right)\leq C (1+t)^n, \ \ \ \ j\geq 0,
$$
to get
\begin{eqnarray} \label{eddd}
&&\|E(k, j, s)\|^2_{L^1(4(1+t)2^j B)}\nonumber\\
&=&\left\|\left(\int_0^{\infty}\!\!\!\!\int_{\substack{  d(x,y)<\tau^{1/m}}}
|F_{\tau,s}(L)P_{U_{j}(B)} (r^{-mk}L^{M-k} b)(y)|^2 {d\mu(y)\over V(x,\tau^{1/m})}
{d\tau\over \tau}\right)^{1/2}\right\|^2_{L^1(4(1+t)2^j B)}\nonumber\\
&\leq& \left\|\left(\int_0^{\infty}\!\!\!\!\int_{\substack{  d(x,y)<\tau^{1/m}}}
|F_{\tau,s}(L)P_{U_{j}(B)} (r^{-mk}L^{M-k} b)(y)|^2
{d\mu(y)\over V(y,\tau^{1/m})}{d\tau\over \tau}\right)^{1/2}\right\|^2_{L^2(4(1+t)2^j B)} \mu(4(1+t)2^j B)\nonumber\\
&\leq& \left\|(1-e^{-s^mL})^{M}G_{k,r,M}(L)F(L)P_{U_{j}(B)} (r^{-mk}L^{M-k} b)(y)\right\|_2^2 \mu(4(1+t)2^j B)\nonumber\\
&\leq& C (\|G_{k,r,M}\|_{L^\infty}+ \|F\|_{L^\infty}) r^{-2mk}\left\|P_{U_{j}(B)} L^{M-k}b\right\|^2_{2}\mu((1+t)2^j B)\nonumber\\
&\leq& C 2^{-2j \varepsilon} \mu(2^j B)^{-1}\mu((1+t)2^j B)\nonumber\\
&\leq& C 2^{-2j \varepsilon} (1+t)^{n}.
\end{eqnarray}

Next we show that for some $\varepsilon'>0$,
\begin{eqnarray} \label{eddde}
\|E(k, j, s)\|_{L^1((4(1+t)2^j B)^c)}   \leq  C2^{-j \varepsilon'}(1+t)^{n/2},
\end{eqnarray}
 and this is the major one.
We have a decomposition according to the frequency,
\begin{eqnarray*}
E(k, j, s)&=&\left(\int_0^{\infty}\!\!\!\!\int_{\substack{  d(x,y)<\tau^{1/m}}}
|F_{\tau,s}(L)P_{U_{j}(B)} (r^{-mk}L^{M-k} b)(y)|^2 {d\mu(y)\over V(x,\tau^{1/m})}{d\tau\over \tau}\right)^{1/2}\\
&=& \left(\sum_{\ell\in\ZZ}\int_{2^{-\ell}}^{2^{-\ell+1}}\!\!\!\!\int_{\substack{  d(x,y)<\tau^{1/m}}}
|F_{\tau,s}(L)P_{U_{j}(B)} (r^{-mk}L^{M-k} b)(y)|^2 {d\mu(y)\over V(x,\tau^{1/m})}{d\tau\over \tau}\right)^{1/2}\\
&\leq& \sum_{\ell\in\ZZ}\left(\int_{2^{-\ell}}^{2^{-\ell+1}}\!\!\!\!\int_{\substack{  d(x,y)<\tau^{1/m}}}
|F_{\tau,s}(L)P_{U_{j}(B)} (r^{-mk}L^{M-k} b)(y)|^2 {d\mu(y)\over V(x,\tau^{1/m})}{d\tau\over \tau}\right)^{1/2}\\
&=:&\sum_{\ell\in\ZZ} E(k, j, s, \ell).
\end{eqnarray*}
If $\ell>0$, let $\nu^+_0\in {\mathbb Z}_+$ be a  positive integer such that
\begin{eqnarray}\label{bbb}
 2\leq 2^{\nu^+_0+j-\ell(m-1)/m}r\leq 4 \ \ \ \ {\rm if} \ \ \ 2^{j-\ell(m-1)/m}r<1;\nonumber\\[1pt]
 \nu^+_0=1 \ \ \ \ {\rm if} \ \ \ 2^{j-\ell(m-1)/m}r\geq 1.
 \end{eqnarray}
If $\ell\leq 0$, let $\nu^-_0\in {\mathbb Z}_+$ be a  positive integer such that
\begin{eqnarray}\label{bbb-}
 2\leq 2^{\nu^-_0+j+\ell/m}r\leq 4 \ \ \ \ {\rm if} \ \ \ 2^{\ell/m+j}r<1;\nonumber\\[1pt]
 \nu^-_0=1 \ \ \ \ {\rm if} \ \ \ 2^{\ell/m+j}r\geq 1.
 \end{eqnarray}
 Then
\begin{eqnarray*}
\|E(k, j, s)\|_{L^1((4(1+t)2^j B)^c)}
&\leq& \sum_{\ell>0}\sum_{\nu\geq \nu^+_0} \|E(k, j, s, \ell)\|_{L^1({U_{\nu+j}((1+t)B)} )}+
 \sum_{\ell>0}\|E(k, j, s, \ell)\|_{L^1(B(x_B,2(1+t)2^{\ell(m-1)/m}))}\\
 &+& \sum_{\ell\leq 0}\sum_{\nu\geq \nu^-_0} \|E(k, j, s, \ell)\|_{L^1({U_{\nu+j}((1+t)B)} )}+
 \sum_{\ell\leq 0}\|E(k, j, s, \ell)\|_{L^1(B(x_B,2(1+t)2^{-\ell/m}))}\\
&=:&I^+(k, j, s)+II^+(k, j, s)+I^-(k, j, s)+II^-(k, j, s).
\end{eqnarray*}
We first estimate terms $II^+(k, j, s)$ and $I^+(k, j, s)$.
Note that there is no term $II^+(k, j, s)$ if $2^{j-\ell(m-1)/m}r\geq 1$ and $\ell>0$. When $2^{j-\ell(m-1)/m}r\leq 1$ and $\ell>0$,
for the   term $II^+(k, j, s)$,
we note that from the doubling condition
$$
V(x,\tau^{1/m})\sim V(y,\tau^{1/m}), \mbox{\,when\,} d(x,y)<\tau^{1/m}
$$
 and then it follows from estimate  \eqref{ee3.0} that
\begin{eqnarray*}
&&\|E(k, j, s, \ell)\|^2_{L^2(X)} \\
&=&\int_X\int_{2^{-\ell}}^{2^{-\ell+1}}\!\!\!\!\int_{\substack{  d(x,y)<\tau^{1/m}}}
|F_{\tau,s}(L)P_{U_{j}(B)} (r^{-mk}L^{M-k} b)(y)|^2 {d\mu(y)\over V(x,\tau^{1/m})}{d\tau\over \tau}\,d\mu(x)\\
&=&\int_{2^{-\ell}}^{2^{-\ell+1}}\int_X|F_{\tau,s}(L)P_{U_{j}(B)} (r^{-mk}L^{M-k} b)(y)|^2\int_{\substack{  d(x,y)<\tau^{1/m}}}\,d\mu(x)
 {d\mu(y)\over V(y,\tau^{1/m})}{d\tau\over \tau}\\
&\leq&\int_{2^{-\ell}}^{2^{-\ell+1}}\|F_{\tau,s}(L)P_{U_{j}(B)} (r^{-mk}L^{M-k} b)\|_2^2{d\tau\over \tau}\\
&\leq&\int_{2^{-\ell}}^{2^{-\ell+1}} \|F_{\tau,s}\|_{L^\infty}^2 \|P_{U_{j}(B)} (r^{-mk}L^{M-k} b)\|_2^2{d\tau\over \tau}\\
&\leq&\int_{2^{-\ell}}^{2^{-\ell+1}}C\min\{1, (\tau^{-1/m} r)^{2M}\} \tau^{n} 2^{-j 2\varepsilon} V(x_B,2^j r)^{-1}{d\tau\over \tau}\\
&\leq&C\min\{1, (2^{\ell/m} r)^{2M}\} 2^{-n\ell} 2^{-j 2\varepsilon} V(x_B,2^j r)^{-1}
\end{eqnarray*}
and thus
\begin{eqnarray*}
&&II^+(k, j, s)\\
&\leq& \sum_{\ell>0}\|E(k, j, s, \ell)\|_{L^2(B(x_B,(1+t)2^{\ell(m-1)/m}))} V\big(x_B,2(1+t)2^{\ell(m-1)/m}\big)^{1/2}\\
&\leq& \sum_{\ell>0}C\min\{1, (2^{\ell/m} r)^M\} 2^{-n\ell/2} 2^{-j \varepsilon} V(x_B,2^j r)^{-1/2}V(x_B,(1+t)2^{\ell(m-1)/m})^{1/2}\\
&\leq& C2^{-j \varepsilon}\sum_{\ell>0}\min\{1, (2^{\ell/m} r)^M\} (2^{\ell/m}r)^{-n/2}(1+t)^{n/2}\\
&\leq& C2^{-j \varepsilon}(1+t)^{n/2}.
\end{eqnarray*}
To estimate  term $I^+(k, j, s)$, we first note that it follows from \eqref{bbb} that for $\tau\in [2^{-\ell},2^{-\ell+1}]$
and $\ell>0$,
$$
\tau^{1/m}\leq 2^{1/m}2^{-\ell/m}\leq 1  \leq   2^{\ell(m-1)/m}\leq
 2^{\nu+j-1}(1+t)r.
$$
So if $d(x,y)<\tau^{1/m}$ and $x\in {U_{\nu+j}((1+t)B)} $, then $y\in U'_{\nu+j}((1+t)B)$ where
$$
U'_{\nu+j}((1+t)B):=B(x_B, 2^{\nu+j+1}(1+t)r)\backslash B(x_B, 2^{\nu+j-1}(1+t)r).
$$
Then we have
\begin{eqnarray}\label{I+}
&&\|E(k, j, s, \ell)\|^2_{L^2({U_{\nu+j}((1+t)B)} )}\nonumber \\
&=&\int_{{U_{\nu+j}((1+t)B)} }\int_{2^{-\ell}}^{2^{-\ell+1}}\!\!\!\!\int_{\substack{  d(x,y)<\tau^{1/m}}}
|F_{\tau,s}(L)P_{U_{j}(B)} (r^{-mk}L^{M-k} b)(y)|^2 {d\mu(y)\over V(x,\tau^{1/m})}{d\tau\over \tau}\,d\mu(x)\nonumber\\
&=&\int_{2^{-\ell}}^{2^{-\ell+1}}\int_{U'_{\nu+j}((1+t)B)}|F_{\tau,s}(L)P_{U_{j}(B)} (r^{-mk}L^{M-k} b)(y)|^2\int_{\substack{  d(x,y)<\tau^{1/m}}}\,d\mu(x)
 {d\mu(y)\over V(y,\tau^{1/m})}{d\tau\over \tau}\nonumber\\
&\leq&\int_{2^{-\ell}}^{2^{-\ell+1}}\|F_{\tau,s}(L)P_{U_{j}(B)} (r^{-mk}L^{M-k} b)\|_{L^2(U'_{\nu+j}((1+t)B))}^2
{d\tau\over \tau}\nonumber\\
&\leq&\int_{2^{-\ell}}^{2^{-\ell+1}} \|F_{\tau,s}(L)\|_{L^2(U_{j}(B))\to L^2(U'_{\nu+j}((1+t)B))}^2 \|P_{U_{j}(B)}
(r^{-mk}L^{M-k} b)\|_2^2{d\tau\over \tau}
\end{eqnarray}
By Proposition~\ref{le2.3},
\begin{eqnarray}\label{2to2offinproof}
\|P_{U'_{\nu+j}((1+t)B)} F_{\tau,s}(L)P_{U_{j}(B)}\|_{2\to 2}\leq C \left(2^\nu \tau^{-1/m} (1+t)r\right)^{-\alpha}
\|\delta_{\tau^{-1}} F_{\tau,s}\|_{{{\bf B}}^{\alpha } }
\end{eqnarray}
 for every $\alpha> 0$.
To go on, we claim  that for every $\tau>0$
\begin{eqnarray}\label{claimBw}
\|\delta_{\tau^{-1}} F_{\tau,s}\|_{{\bf B}^{\alpha } }
&=& \int_{-\infty}^{\infty} |\widehat{\delta_{\tau^{-1}} F_{\tau,s}}(\xi)|(1+|\xi|)^{\alpha}d\xi\nonumber\\
&\leq&
C \max\{1,\tau^{ n/2-\alpha}\}(1+t)^\alpha \min\{1, (\tau^{-1/m} r)^M\}.
\end{eqnarray}
Let us show the claim \eqref{claimBw}. Recall that
 $\psi\in C_c^\infty$   supported in  $\psi\subset [1/8, 2]$ and $\psi(\lambda)=1$ for $\lambda\in [1/4,1]$.
We have that  for $\tau>0$,
\begin{eqnarray*}
\|\delta_{\tau^{-1}} F_{\tau,s}\|_{{\bf B}^{\alpha } }
&=& \| \phi(\lambda)(1-e^{-s^m\tau^{-1}\lambda})^{M}G_{k,r,M}(\tau^{-1}\lambda)F(\tau^{-1}\lambda)\|_{{\bf B}^{\alpha } }\\
&\leq& \| \psi(\lambda)(1-e^{-s^m\tau^{-1}\lambda})^{M}\|_{{\bf B}^{\alpha } }
\| \psi(\lambda)G_{k,r,M}(\tau^{-1}\lambda)\|_{{\bf B}^{\alpha } }
\| \phi(\lambda)F(\tau^{-1}\lambda)\|_{{\bf B}^{\alpha } }\\
&\leq& C\| \psi(\lambda)(1-e^{-s^m\tau^{-1}\lambda})^{M}\|_{C^{\alpha+2}}
\| \psi(\lambda)G_{k,r,M}(\tau^{-1}\lambda)\|_{C^{\alpha+2}}
\| \phi(\lambda)F(\tau^{-1}\lambda)\|_{{\bf B}^{\alpha } }.
\end{eqnarray*}
Note that for every $s\in [r, \sqrt[m]{2}r]$,
$$
\| \psi(\lambda)(1-e^{-s^m\tau^{-1}\lambda})^{M}\|_{C^{\alpha+2}}\leq C\min\{1,(\tau^{-1/m}r)^M\}
$$
and
$$
\| \psi(\lambda)G_{k,r,M}(\tau^{-1}\lambda)\|_{C^{\alpha+2}}\leq C
$$
with $C$ independent of $k,\tau$ and $r$. Let us  estimate $\| \phi(\lambda)F(\tau^{-1}\lambda)\|_{{\bf B}^{\alpha } }$.
It follows from the Fourier transform ${\mathcal F} \big(\phi F(\tau^{-1}\cdot)\big)$ of $\phi F(\tau^{-1}\cdot)$ that
\begin{eqnarray*}
{\mathcal F} \big(\phi F(\tau^{-1}\cdot)\big)(\xi)=\int_{\mathbb{R}} \phi(\lambda)\frac{e^{i( \tau^{-1} t-\xi)\lambda}}{(1+\tau^{-1}\lambda)^{n/2}}d\lambda.
\end{eqnarray*}
Integration by parts gives for every $N\in \mathbb{N}$,
\begin{eqnarray*}
\left|{\mathcal F} \big(\phi F(\tau^{-1}\cdot)\big)(\xi)\right|\leq
C_N\min\{1,\tau^{ n/2}\}(1+|\tau^{-1} t-\xi|)^{-N},
\end{eqnarray*}
which yields
\begin{eqnarray*}
\| \phi(\lambda)F(\tau^{-1}\lambda)\|_{{\bf B}^{\alpha } }&\leq& C\min\{1,\tau^{ n/2}\}\int_{\mathbb{R}} (1+|\tau^{-1}t-\xi|)^{-N}
(1+|\xi|)^\alpha d\xi\\
&\leq& C\max\{1,\tau^{n/2-\alpha}\}\left(1+ t\right)^\alpha.
\end{eqnarray*}
Hence,  for every $\tau>0$,
\begin{eqnarray*} 
\|\delta_{\tau^{-1}} F_{\tau,s}\|_{{\bf B}^{\alpha } }
&\leq&
C \max\{1,\tau^{ n/2-\alpha}\}(1+t)^\alpha \min\{1, (\tau^{-1/m} r)^M\}.
\end{eqnarray*}
This proves our claim  \eqref{claimBw}.

Next letting $\alpha$ be a fixed number such that $0<\alpha-n/2<\varepsilon$,  we apply   \eqref{ee3.0}
and  the doubling condition \eqref{eq2.2} and \eqref{I+}, \eqref{2to2offinproof} and \eqref{claimBw} to get
\begin{eqnarray}\label{kkkk}
 &&I^+(k, j, s)\nonumber\\
 &\leq&  \sum_{\ell>0}\sum_{\nu\geq \nu^+_0} \|E(k, j, s, \ell)\|_{L^2({U_{\nu+j}((1+t)B)} )}V(x_B,2^{\nu+j}(1+t)r)^{1/2}\nonumber\\
  &\leq&  \sum_{\ell>0}\sum_{\nu\geq \nu^+_0} \left(\int_{2^{-\ell}}^{2^{-\ell+1}}
  \|F_{\tau,s}(L)\|_{L^2(U_{j}(B))\to L^2(U'_{\nu+j}((1+t)B))}^2 \|P_{U_{j}(B)}
(r^{-mk}L^{M-k} b)\|_2^2{d\tau\over \tau}\right)^{1/2}V(x_B,2^{\nu+j}(1+t)r)^{1/2}\nonumber\\
&\leq&  \sum_{\ell>0}\sum_{\nu\geq \nu^+_0} \left(2^\nu 2^{\ell/m} (1+t)r\right)^{-\alpha}2^{(\alpha-n/2) \ell}(1+t)^\alpha \min\{1, (2^{\ell/m} r)^M\}
2^{-j \varepsilon}
\left({V(x_B,2^{\nu+j}(1+t)r) \over V(x_B, 2^jr)}\right)^{1/2}\nonumber\\
  &\leq& C2^{-j \varepsilon}(1+t)^{n/2}\sum_{\ell>0}\sum_{\nu\geq \nu^+_0}
  2^{(-\alpha+n/2)\nu} ( 2^{\ell/m} r )^{-\alpha} 2^{(\alpha-n/2) \ell}\min\{1, (2^{\ell/m} r)^M\}
  \nonumber\\
  &\leq & C2^{-j \varepsilon}(1+t)^{n/2}\sum_{\ell>0} (2^{j-\ell(m-1)/m}r)^{(\alpha-n/2)}
  ( 2^{\ell/m} r )^{-\alpha} 2^{(\alpha-n/2) \ell}\min\{1, (2^{\ell/m} r)^M\} \nonumber\\
  &\leq& C2^{-j (\varepsilon-\alpha+n/2)}(1+t)^{n/2}\sum_{\ell>0}  (2^{\ell/m}r )^{-n/2}\min\{1, (2^{\ell/m} r)^M\} \nonumber\\
  &\leq& C2^{-j \varepsilon'}(1+t)^{n/2} \nonumber
\end{eqnarray}
with $\epsilon'=\epsilon-\alpha+n/2.$

For terms $I^-(k, j, s)$ and $II^-(k, j, s)$, the estimates are similar to terms  $I^+(k, j, s)$ and $II^+(k, j, s)$ and simpler.
Note that there is no term $II^-(k, j, s)$ if $2^{j+\ell/m}r\geq 1$ and $\ell\leq 0$. So when $2^{j+\ell/m}r\leq 1$ and $\ell\leq 0$,
then for the   term $II^-(k, j, s)$,
we note that it follows from estimate  \eqref{ee3.0} that
\begin{eqnarray*}
&&\|E(k, j, s, \ell)\|^2_{L^2(X)} \\
&=&\int_X\int_{2^{-\ell}}^{2^{-\ell+1}}\!\!\!\!\int_{\substack{  d(x,y)<\tau^{1/m}}}
|F_{\tau,s}(L)P_{U_{j}(B)} (r^{-mk}L^{M-k} b)(y)|^2 {d\mu(y)\over V(x,\tau^{1/m})}{d\tau\over \tau}\,d\mu(x)\\
&=&\int_{2^{-\ell}}^{2^{-\ell+1}}\int_X|F_{\tau,s}(L)P_{U_{j}(B)} (r^{-mk}L^{M-k} b)(y)|^2\int_{\substack{  d(x,y)<\tau^{1/m}}}\,d\mu(x)
 {d\mu(y)\over V(y,\tau^{1/m})}{d\tau\over \tau}\\
&\leq&\int_{2^{-\ell}}^{2^{-\ell+1}}\|F_{\tau,s}(L)P_{U_{j}(B)} (r^{-mk}L^{M-k} b)\|_2^2{d\tau\over \tau}\\
&\leq&\int_{2^{-\ell}}^{2^{-\ell+1}} \|F_{\tau,s}\|_{L^\infty}^2 \|P_{U_{j}(B)} (r^{-mk}L^{M-k} b)\|_2^2{d\tau\over \tau}\\
&\leq&\int_{2^{-\ell}}^{2^{-\ell+1}}C\min\{1, (\tau^{-1/m} r)^{2M}\}  2^{-j 2\varepsilon} V(x_B,2^j r)^{-1}{d\tau\over \tau}\\
&\leq&C\min\{1, (2^{\ell/m} r)^{2M}\}  2^{-j 2\varepsilon} V(x_B,2^j r)^{-1}
\end{eqnarray*}
and thus
\begin{eqnarray*}
&&II^-(k, j, s)\\
&\leq& \sum_{\ell\leq0}\|E(k, j, s, \ell)\|_{L^2(B(x_B,(1+t)2^{-\ell/m}))} V\big(x_B,(1+t)2^{-\ell/m}\big)^{1/2}\\
&\leq& \sum_{\ell\leq0}C\min\{1, (2^{\ell/m} r)^M\} 2^{-j \varepsilon} V(x_B,2^j r)^{-1/2}V(x_B,(1+t)2^{-\ell/m})^{1/2}\\
&\leq& C2^{-j \varepsilon}\sum_{\ell\leq 0}\min\{1, (2^{\ell/m} r)^M\} (2^{\ell/m}r)^{-n/2}(1+t)^{n/2}\\
&\leq& C2^{-j \varepsilon}(1+t)^{n/2}.
\end{eqnarray*}
To estimate  term $I^-(k, j, s)$, we first note that it follows from \eqref{bbb-} that for $\tau\in [2^{-\ell},2^{-\ell+1}]$
$$
\tau^{1/m}\leq 2^{1/m}2^{-\ell/m}\leq  2^{\nu+j-1}(1+t)r.
$$
So if $d(x,y)<\tau^{1/m}$ and $x\in {U_{\nu+j}((1+t)B)} $, then $y\in U'_{\nu+j}((1+t)B)$ where
$$
U'_{\nu+j}((1+t)B):=B(x_B, 2^{\nu+j+1}(1+t)r)\backslash B(x_B, 2^{\nu+j-1}(1+t)r).
$$
Then , letting $\alpha$ be a fixed number such that $0<\alpha-n/2<\varepsilon$, by \eqref{claimBw} for $\ell\leq 0$ and similarly as in \eqref{kkkk},
\begin{eqnarray}\label{kkkk-}
 &&I^-(k, j, s)\nonumber\\
 &\leq&  \sum_{\ell\leq 0}\sum_{\nu\geq \nu^-_0} \|E(k, j, s, \ell)\|_{L^2({U_{\nu+j}((1+t)B)} )}V(x_B,2^{\nu+j}(1+t)r)^{1/2}\nonumber\\
  &\leq&  \sum_{\ell\leq 0}\sum_{\nu\geq \nu^-_0} \left(\int_{2^{-\ell}}^{2^{-\ell+1}}
  \|F_{\tau,s}(L)\|_{L^2(U_{j}(B))\to L^2(U'_{\nu+j}((1+t)B))}^2 \|P_{U_{j}(B)}
(r^{-mk}L^{M-k} b)\|_2^2{d\tau\over \tau}\right)^{1/2}V(x_B,2^{\nu+j}(1+t)r)^{1/2}\nonumber\\
&\leq&  \sum_{\ell\leq 0}\sum_{\nu\geq \nu^-_0} \left(2^\nu 2^{\ell/m} (1+t)r\right)^{-\alpha}(1+t)^\alpha \min\{1, (2^{\ell/m} r)^M\}
2^{-j \varepsilon}
\left({V(x_B,2^{\nu+j}(1+t)r) \over V(x_B, 2^jr)}\right)^{1/2}\nonumber\\
  &\leq& C2^{-j \varepsilon}(1+t)^{n/2}\sum_{\ell\leq 0}\sum_{\nu\geq \nu^-_0}  2^{(-\alpha+n/2)\nu} ( 2^{\ell/m} r )^{-\alpha} \min\{1, (2^{\ell/m} r)^M\}
  \nonumber\\
&\leq & C2^{-j \varepsilon}(1+t)^{n/2}\sum_{\ell\leq 0} ( 2^{\ell/m} r)^{-\alpha} \min\{1, (2^{\ell/m} r)^M\}\nonumber\\
&\leq& C2^{-j \varepsilon}(1+t)^{n/2}\nonumber.
\end{eqnarray}

Combining two estimates of $I^+(k, j, s)$, $II^+(k, j, s)$, $I^-(k, j, s)$ and $II^-(k, j, s)$ we obtain \eqref{eddde}.
This, in combination with \eqref{eddd} and \eqref{e1}, shows that
\begin{eqnarray*}
&& \sum_{k=0}^{M-1} \left\|\left(\int_0^{\infty}\!\!\!\!\int_{\substack{  d(x,y)<\tau^{1/m}}}
|E_k(y)|^2 {d\mu(y)\over V(x,\tau^{1/m})}{d\tau\over \tau}\right)^{1/2}\right\|_{L^1} \\
&\leq&  C\sum_{k=0}^{M-1} \sum_{j\geq 0}\sup_{s\in [r,\sqrt[m] 2 r]} \|E(k, j, s)\|_{L^1(X)}
   \leq  C\sum_{j\geq 0}2^{-j \varepsilon'}(1+t)^{n/2} \leq C (1+t)^{n/2}.
\end{eqnarray*}

\medskip

\noindent
{\bf Case 2.} \ $k=M$.

In this case, we write
  \begin{eqnarray*}\label{e11000}
&&\left\|\left(\int_0^{\infty}\!\!\!\!\int_{\substack{  d(x,y)<\tau^{1/m}}}
|E_M(y)|^2 {d\mu(y)\over V(x,\tau^{1/m})}{d\tau\over \tau}\right)^{1/2}\right\|_{L^1} \\
&\leq& C \sum_{j=1}^{(2M-1)M}\left\|\left(\int_0^{\infty}\!\!\!\!\int_{\substack{  d(x,y)<\tau^{1/m}}}
|(1-e^{-r^mL})^{M}e^{-j r^mL} F_\tau(L) (r^{-mM}  b)(y)|^2 {d\mu(y)\over V(x,\tau^{1/m})}{d\tau\over \tau}\right)^{1/2}\right\|_{L^1}.
\end{eqnarray*}
Similar to the proof of $E_k$ as in {\bf Case 1 }, we have that
 \begin{eqnarray*}\label{e110000}
&&\left\|\left(\int_0^{\infty}\!\!\!\!\int_{\substack{  d(x,y)<\tau^{1/m}}}
|E_M(y)|^2 {d\mu(y)\over V(x,\tau^{1/m})}{d\tau\over \tau}\right)^{1/2}\right\|_{L^1} \\
&\leq&  C\sum_{j\geq 0}2^{-j \varepsilon'}(1+t)^{n/2} \leq C (1+t)^{n/2}.
\end{eqnarray*}
 Hence, we have proved
estimate (\ref{e3.3}), and then concluded  the proof of \eqref{e1.66}.

 \medskip

 Now we turn to prove \eqref{e1.555}. To do this, we need to state a complex interpolation result.
Fix a pair of  Banach spaces $E_0, E_1$ continuously embedded in some Banach space $V$ such that
$E_0\cap E_1$ contains a dense subspace ${\mathcal D}$ of both $E_0, E_1$ under the corresponding norms.
Let $S=\{z: 0<{\rm Re} z<1\}$ and ${\bar S}=\{z: 0\leq {\rm Re} z \leq 1\}$. Following \cite{CA}, we define ${\mathcal F}(E_0, E_1)$
to be the set of all functions $F$ on ${\bar S}$ with values in $E_0+ E_1$, analytic in $S$ and such that
$F(it)\in E_0$ is $E_0$-continuous and tends to $0$ as $|t|\to \infty$
and $F(1+it)\in E_1$ is $E_1$-continuous and tends to $0$  as $|t|\to \infty$.
${\mathcal F}(E_0, E_1)$ becomes a Banach space under the norm
$$
\|F\|_{{\mathcal F}}= \sup_{t\in{\mathbb R}} \max \big( \|F(it)\|_{E_0}, \  \|F(1+it)\|_{E_1} \big).
$$
Given a real number $\theta, 0<\theta<1$, Calder\'on constructed a subspace $[E_0, E_1]_{\theta}$ of
$E_0 +E_1$ as follows:
$$
[E_0, E_1]_{\theta} =\{F(\theta): F\in {\mathcal F}(E_0, E_1) \}.
$$
By introducing the norm
$$\|F\|_{[E_0, E_1]_{\theta}}= \inf \{ \|F\|_{\mathcal F}:  F\in {\mathcal F}(E_0, E_1) , \ F(\theta)=f  \},
$$
$[E_0, E_1]_{\theta}$ becomes a Banach space continuously embedded in $E_0+ E_1.$
We next define analytic families of operators. Let $\{T_z\}$ be a family of linear operators indexed
by $z\in {\bar S}$ so that for each $z$, $T_z$ is a mapping of functions in ${\mathcal D}$ to measurable
functions on $E$. Following \cite{SA}, $\{ T_z\}$ is called an analytic family if for any $g\in{\mathcal D}$
and for almost all $y\in E$, $(T_z(g)) (y)$ is analytic in $S$ and continuous on ${\bar S}$. The analytic family
$\{T_z\}$ is of admissible growth if for all $y\in {\mathcal D}$ there exists a constant $C_{\theta}$ and a constant
$a\in\pi$ such that
$$
\sup_{z\in{\bar S}} {\rm log} |(T_z g)(y)| \leq C_{\theta} e^{a|{\rm Im\,  z|}}
$$
for almost all $y\in E$. Then we have the following result, for its proof, we refer it to  \cite{SA}, \cite[Theorem 3]{G}.

\begin{lemma}\label{le8.8}
Let $E_0, E_1, {\mathcal D}$ as before, $0<p_0, p_1\leq \infty$, and let
$\{ T_z\}$  be an analytic family of linear operators which is of admissible growth. If for all
$f\in {\mathcal D}$ $\|T_zf\|_{L^{p_j}}\leq c_j(z)\|f\|_{E_j}$ when ${\rm Re} z=j, \, j=0, 1$ for some constants
$c_j(z)$ that satisfy
${\rm log} c_j(z)\leq A  e^{a|{\rm Im\, z|}}, A>0, 0\leq a<\pi
$, then for all $z\in S$ there exist $A_z>0$ such that
for $f\in {\mathcal D}$,
 $$
 \|T_zf \|_{L^p}\leq A_z \|f\|_{[E_0, E_1]_{\theta}}, \ \ \ \  {\rm when} \ \ \ \ {\rm Re z}=\theta
 $$
 where
 $$
 {1\over p}={1-\theta\over p_0} + {\theta\over p_1}.
 $$
 \end{lemma}

 We now apply Lemma~\ref{le8.8} to prove \eqref{e1.555}. Let $E_0=H^1_L(X)$, $E_1=H^2(X)$ and  ${\mathcal D}=H^1_L(X) \cap H^2(X)$,
 and ${\mathcal D}$ is dense in both $H^1(X)$. Since
 $$
 [H^1_L(X), H^2(X)]_{\theta}= H_L^p(X),\ \ \ \ \  {1\over p}=1-{\theta\over 2}.
 $$
 Consider the analytic family of operators
$$
T_z:=e^{z^2}(1+|t|)^{-zn/2}(1+L)^{-zn/2}e^{itL}, \ \ \ \  0\leq {\rm Re}\, z\leq 1.
$$
Note that $T_z$ is a holomorphic function of $z$ in the sense that
$$
z\to \int_X T_z f(x)g(x)d\mu(x)
$$
for $f,g\in L^2(X)$.
If $y\in{\mathbb R}$, then
$$
T_{iy}(L)=e^{-y^2}(1+|t|)^{-iyn/2}(1+L)^{-iny/2}e^{itL}.
 $$
Since $H^2_L(X)\subseteq L^2(X)$, we have
\begin{eqnarray*}
\|T_{iy}(f)\|_{L^2(X)}&=&\|e^{-y^2}(1+|t|)^{-iyn/2}(1+L)^{-iny/2}e^{itL}f\|_2\\
&\leq& C\|(1+\lambda)^{-iny/2}e^{it\lambda}\|_{L^\infty}\|f\|_2\\
&\leq& C\|f\|_{L^2(X)}  \leq C\|f\|_{H_L^2(X)}
\end{eqnarray*}
with $C$ independent of $t$ and $y$.
On the other hand,  it follows from Proposition~\ref{prop2.7} that
$$
\|(1+L)^{-iny/2}\|_{H_L^1(X)\to H^1_L(X)}\leq C(1+|y|)^{n/2+1}.
$$
This, together with    \eqref{e1.6}, shows that $(1+L)^{-n/2}e^{itL}$ is   bounded from
$H^1_L(X)$ to $L^1(X)$ and
$$
\|T_{1+iy}(f)\|_{L^1(X)}\leq Ce^{1-y^2}(1+|t|)^{-n/2}\|(1+L)^{-n/2}e^{itL}\|_{H_L^1(X)\to L^1(X)} \|(1+L)^{-iny/2}f\|_{H_L^1(X)}
\leq C\|f\|_{H_L^1(X)}
$$
with $C$ independent of $t$ and $y$.
Then by   Lemma~\ref{le8.8}, we have that for $\theta=1-2/p$
and $s=n(1/2-1/p)$
$$
\|(1+L)^{-s}e^{itL}f\|_{L^p(X)}=\|e^{-\theta^2}(1+|t|)^{\theta n/2}T_\theta f\|_{L^p(X)}\leq C(1+|t|)^{s}\|f\|_{H_L^p(X)}
$$
as desired for $1<p\leq 2$. This proves \eqref{e1.555}.

 By duality, estimate \eqref{e1.55522} holds for $2<p<\infty$. This completes the proof of Theorem~\ref{th1.1}.
\hfill{}$\Box$

\medskip
 \section{Application:  Schr\"odinger  groups for the Kohn Laplacian}
\setcounter{equation}{0}
In this section, we give an application of Theorem~\ref{th1.1} to the Kohn Laplacian $\Box_b$ on polynomial model
domains treated by Nagel-Stein~\cite{NS}.

Let $M$ be the boundary of an unbounded polynomial domain $\Omega:= \{(z,w)\in\mathbb{C}^2:\ {\rm Im}(w)>P(z)\}$, where $P$ is a
real, subharmonic, nonharmonic polynomial of degree $m$ (see \cite{NS}).
Let $\overline{\partial}_b$ be the tangential Cauchy-Riemann operator on $M$ which maps functions to $(0, 1)$-forms,
and let $\big(\overline{\partial}_b\big)^{\ast}$ be the formal
adjoint which maps $(0, 1)$-forms to functions. As in \cite{NS}, choose real vector fields $X_1, X_2$ on $M$ so that
we can identify $\overline{\partial}_b f$ with
\begin{align*}
(X_1 + i X_2) f
\end{align*}
by identifying functions and $(0,1)$ forms on $M$. Then
we define Kohn Laplacian $\Box_b$ acting on functions by
$\Box_b:= \big(\overline{\partial}_b\big)^{\ast}\overline{\partial}_b$.
Since $\Box_b$ is a self-adjoint operator, it admits a spectral decomposition $E(\lambda)$;
in particular, $E(0) = \pi$, where $\pi$ is the Szeg\"{o} projection from $L^2(M)$ to the null space of $\Box_b$.
It is known (see \cite{Str}) that the heat kernel $ K_{e^{-s\Box_b}} (x, y)$ of $e^{-s\Box_b}$   is in terms of Carnot-Carath\'eodory distance $d$ on $M$,
and there  exist two positive constants $C$ and $c$ such that
\begin{eqnarray}\label{e7.2}
\left|
K_{e^{-s\Box_b}} (x, y)\right| \leq \frac{C}{V(x, d(x,y))} \exp\left(-c{d(x,y)^{2}/ s}\right),
\end{eqnarray}
where $V(x, \delta)$ denotes the volume of ball of radius $\delta$ in the $d$ metric, centered at $x$.
Note that there exist $C$ and $Q$ such that
\begin{eqnarray}\label{qqq}
V(x, \lambda \delta)\leq C\lambda^Q V(x,  \delta), \ \ \ \ \lambda\geq 1,
\end{eqnarray}
and  so
\begin{eqnarray}\label{qqqq}
V(x, r)\leq C\left(1+ {d(x,y)\over r}\right)^Q V(y, r),
\end{eqnarray}
uniformly for all $x, y\in M$ and $r>0.$
It is worth pointing out that the heat kernels $K_{e^{-s\Box_b}} (x, y)$
of the Kohn  Laplacian $\Box_b$ on the boundary $M$ do not satisfy standard Gaussian
upper bounds \eqref{GE} with $m=2$. However, the Kohn Laplacian $\Box_b$ satisfies the finite speed property of
propagation for the corresponding wave equation(see Theorem 2.3, \cite{Str}). Equivalently, according to \cite[Theorem 2]{Sik}, the Kohn Laplacian $\Box_b$
satisfies $m$-th order Davies-Gaffney  estimates~\eqref{DG} with $m=2$. Thus we would like to apply Theorem~\ref{th1.1}
and relation between $H_{\Box_b}^p(M)$ and $L^p(M)$ to get the endpoint $L^p$-boundedness of the Schr\"odinger group $e^{it\Box_b}$. Before
that we state the following proposition for $\Box_b$, which is used to derive estimates on $L^p(M)$.
\begin{prop}\label{prop7.1} We have the following results:

\begin{itemize}
\item[(1)]
  $F(\Box_b)f=F(\Box_b)(f-\pi f)+F(0)\pi f$, where $\pi$ is the Szeg\"o projection operator and is bounded on $L^p(M)$.

\item[(2)]
 Suppose $m\in \mathcal{S}([0,\infty))$, $m(0)=0$ and $r>0$, where $\mathcal{S}([0,\infty))$ is the Schwarz class on
$[0,\infty)$. Then the kernel of $m(r^2\Box_b)$ satisfies
$$
|K_{m(r^2\Box_b)}(x,y)|\leq C_N \left(1+\frac{\rho(x,y)}{r}\right)^{-N}\frac{1}{V(x,\rho(x,y)+r)}.
$$

\item[(3)]
 Let $\psi\in C_c^\infty(1,4)$ and $\sum_j\psi(2^j\lambda)=1$ for $\lambda>0$. Then there exists $\widetilde\psi\in C_c^\infty(1,4)$ such that for $\pi f=0$
$$
f(x)=\sum_j \psi(2^j\Box_b) \widetilde\psi(2^j\Box_b)f(x).
$$

\item[(4)]
 Define the discrete square function by
$$
G_{d, \Box_b}(f)(x)=\left(\sum_j |\psi(2^j\Box_b)f(x)|^2\right)^{1/2}.
$$
Then for $\pi f=0$,
$$
\|G_{d, \Box_b}(f)\|_{L^p}\sim \|f\|_{L^p}.
$$
\end{itemize}
\end{prop}

\begin{proof}
The properties (1)-(4) are from Proposition 7.4 and estimates (16)--(18) in p. 880 of \cite{Str}.
\end{proof}

Recall that $S_{\Box_b}f$ is the area function of $\Box_b$ given in \eqref{e2.1}.
We have the following result.

\begin{prop}\label{prop7.2}
If $\pi f=0$, then for $1<p<\infty$
$$
\|S_{\Box_b}f\|_{L^p(M)}\leq C\|f\|_{L^p(M)}.
$$
\end{prop}
\begin{proof}
First, if $m,\phi\in \mathcal{S}([0,\infty))$, $|m(\lambda)|+|\phi(\lambda)|\leq C\lambda^\epsilon$ around $\lambda=0$ for some $\epsilon>0$,
then applying (2) of Proposition~\ref{prop7.1}, the kernel of $m(r^2\Box_b)\phi(s^2\Box_b)$ satisfies
\begin{eqnarray}\label{e7.almost}
|K_{m(r^2\Box_b)\phi(s^2\Box_b)}(x,y)|\leq C_N
\left(\min\{\frac{s}{r},\frac{r}{s}\}\right)^{\epsilon} \left(1+\frac{\rho(x,y)}{\max\{s,r\}}\right)^{-N}\frac{1}{V(x,\max\{s,r\})}
\end{eqnarray}
for all $N>0$.

Let $\psi$ and $\widetilde\psi$ be functions in (3) of Proposition~\ref{prop7.1} and thus we have that for $\pi f=0$
\begin{eqnarray}\label{vv}
f(x)=\sum_{j=-\infty}^{\infty} \psi(2^j\Box_b) \widetilde\psi(2^j\Box_b)f(x).
\end{eqnarray}
Define the square function
$$
{G}_{c, \Box_b}(f)(x)= \left(\int_0^\infty   |\psi(t^2\Box_b) f(x)|^2\frac{dt}{t}\right)^{1/2}.
$$
Then  for $\pi f=0$, we apply \eqref{vv}  and \eqref{e7.almost} to obtain
\begin{eqnarray*}
G_{c, \Box_b}(f)(x)&\leq& C\left(\int_0^\infty \sum_j |\psi(t^2\Box_b)\psi(2^j\Box_b) \widetilde\psi(2^j\Box_b)f(x)|^2\frac{dt}{t}\right)^{1/2}\\
&\leq& C\left(\sum_j  \int_{2^{j-1}}^{2^{j+1}} |\psi(t^2\Box_b) \widetilde\psi(2^j\Box_b)\psi(2^j\Box_b)f(x)|^2\frac{dt}{t}\right)^{1/2}\\
&\leq& C\left(\sum_j  \int_{2^{j-1}}^{2^{j+1}} \left|\int_M K_{\psi(t^2\Box_b)
 \widetilde\psi(2^j\Box_b)}(x,y)\psi(2^j\Box_b)f(y)dy\right|^2\frac{dt}{t}\right)^{1/2}\\
&\leq& C\left(\sum_j   \left|{\mathfrak M}(\psi(2^j\Box_b))f(x)\right|^2\right)^{1/2}
\end{eqnarray*}
where    $\mathfrak M$ denotes the Hardy-Littlewood maximal function, that is
 \begin{equation*}
\mathfrak  Mf(x)=\sup_{x\in B} {1\over V(B) }\int_{B}
|f(y)|dy.
\end{equation*}
Hence,
$$
\|G_{c, \Box_b}(f)\|_{L^p}\leq C\left\|\left(\sum_j   \left|{\mathfrak M} (\psi(2^j\Box_b))f(x)\right|^2\right)^{1/2}\right\|_{L^p}
\leq C \|G_{d, \Box_b}(f)\|_{L^p}\leq C\|f\|_{L^p}.
$$

Next directly computation shows that for all $\alpha>0$,
$$
S_{\Box_b}f(x)\leq C_\alpha\left(\int_0^\infty |{\mathfrak M}^*_{\alpha,\Box_b} (f)(x,t)|^2\frac{dt}{t}\right)^{1/2}
$$
where ${\mathfrak M}^*_{\alpha,\Box_b} (f)$ is the Peetre type maximal function (see for example \cite{BPT, Hu}) given by
$$
{\mathfrak M}^*_{\alpha,\Box_b} (f)(x,t):=\sup_{y\in M}\frac{|t^2\Box_b e^{-t^2\Box_b}f(y)|}{(1+t^{-1}d(x,y))^\alpha}.
$$
Then applying \eqref{e7.almost} gives
\begin{eqnarray*}
|t^2\Box_b e^{-t^2\Box_b}f(y)|&\leq& \sum_j |\psi(2^jt^2\Box_b) \widetilde\psi(2^jt^2\Box_b)t^2\Box_b e^{-t^2\Box_b}f(y)|\\
&\leq& \sum_j \int_M \left|K_{\widetilde\psi(2^jt^2\Box_b)t^2\Box_b e^{-t^2\Box_b}}(y,z) \psi(2^jt^2\Box_b) f(z)\right|dz\\
&\leq& \sum_{j\geq 0} C2^{-j\epsilon} \int_M \left(1+\frac{\rho(y,z)}{2^{j/2}t}
\right)^{-N}\frac{1}{V(z,2^{j/2}t)} \left|\psi(2^jt^2\Box_b) f(z)\right|dz\\
&&+\sum_{j<0} C2^{j\epsilon} \int_M \left(1+\frac{\rho(y,z)}{t}\right)^{-N}\frac{1}{V(z,t)} \left|\psi(2^jt^2\Box_b) f(z)\right|dz.
\end{eqnarray*}
Using \eqref{qqqq}, we have
\begin{eqnarray*}
\frac{|t^2\Box_b e^{-t^2\Box_b}f(y)|}{(1+t^{-1}d(x,y))^\alpha}
&\leq& \sum_{j\geq 0} C2^{-j\epsilon} \int_M
\left(1+\frac{\rho(x,z)}{2^{j/2}t}\right)^{-\alpha+Q}\frac{1}{V(x,2^{j/2}t)} \left|\psi(2^jt^2\Box_b) f(z)\right|dz\\
&&+\sum_{j<0} C2^{j\epsilon} \int_M \left(1+\frac{\rho(x,z)}{t}\right)^{-\alpha+Q}\frac{1}{V(x,t)} \left|\psi(2^jt^2\Box_b) f(z)\right|dz\\
&\leq& \sum_{j\geq 0} C2^{-j\epsilon} {\mathfrak M}( \psi(2^jt^2\Box_b) f)(x)+\sum_{j<0} C2^{j\epsilon}  {\mathfrak M}( \psi(2^jt^2\Box_b) f)(x).
\end{eqnarray*}
Now
\begin{eqnarray*}
S_{\Box_b}f(x)&\leq &C_\alpha\left(\int_0^\infty | {\mathfrak M}^*_{\alpha,\Box_b} (f)(x,t)|^2\frac{dt}{t}\right)^{1/2}\\
&\leq& C_\alpha \sum_{j\geq 0} 2^{-j\epsilon}\left(\int_0^\infty |{\mathfrak M}( \psi(2^jt^2\Box_b) f)(x)|^2\frac{dt}{t}\right)^{1/2}\\
&&+C_\alpha \sum_{j< 0} 2^{j\epsilon}\left(\int_0^\infty |{\mathfrak M}( \psi(2^jt^2\Box_b) f)(x)|^2\frac{dt}{t}\right)^{1/2}\\
&\leq& C_\alpha \left(\int_0^\infty |{\mathfrak M}( \psi(t^2\Box_b) f)(x)|^2\frac{dt}{t}\right)^{1/2}.
\end{eqnarray*}
Then for $1<p<\infty$,
$$
\|S_{\Box_b}f\|_{L^p}\leq C \left\|
\left(\int_0^\infty |{\mathfrak M}( \psi(t^2\Box_b) f)(x)|^2\frac{dt}{t}\right)^{1/2}\right\|_{L^p}\leq C \|G_{c, \Box_b} f\|_{L^p}\leq C\|f\|_{L^p}.
$$
The proof of Proposition~\ref{prop7.2} is complete.
\end{proof}

Note that the Kohn Laplacian $\Box_b$ on $M$ satisfies the $m$-th order Davies-Gaffney  estimate~\eqref{DG} with $m=2$.
Recall that $Q$ is   the ``dimension" of $M$ in \eqref{qqq}.
We can apply Theorem~\ref{th1.1} to prove the following result.

\begin{thm}\label{Th5.3}
There exists a  constant $C>0$ independent of $t$ such that for $1<p<\infty$,
\begin{eqnarray}\label{e7.44}
\|(1+\Box_b)^{-s}e^{it\Box_b}f\|_{L^p(M)}\leq C(1+|t|)^{s}\|f\|_{L^p(M)}, \ \ \ t\in{\mathbb R}, \ \
  \ s\geq Q\big|{1\over  2}-{1\over  p}\big|.
\end{eqnarray}
\end{thm}

\begin{proof}
Let  $g=f-\pi f$. We have that $\pi g=0$. Note that $F(\Box_b)f=F(\Box_b)g+F(0)\pi f$.
For $1<p\leq 2$, we  apply  Theorem~\ref{th1.1} and Proposition~\ref{prop7.2} to obtain
\begin{eqnarray*}
\|(1+\Box_b)^{-s}e^{it\Box_b}f\|_{L^p(M)}&\leq& \|(1+\Box_b)^{-s}e^{it\Box_b}g\|_{L^p(M)}+\|\pi f\|_{L^p(M)}\\
&\leq& C(1+|t|)^{s}\|g\|_{H_{\Box_b}^p(M)}+C\| f\|_{L^p(M)}\\
&=& C(1+|t|)^{s}\|S_{\Box_b}g\|_{L^p(M)}+C\| f\|_{L^p(M)}\\
&\leq& C(1+|t|)^{s}\|g\|_{L^p(M)}+C\| f\|_{L^p(M)}\\
&=& C(1+|t|)^{s}\|f-\pi f\|_{L^p(M)}+C\| f\|_{L^p(M)}\\
&\leq& C(1+|t|)^{s}\|f\|_{L^p(M)}.
\end{eqnarray*}
By duality, we have the result \eqref{e7.44} for $p>2$. This completes the proof of  Theorem~\ref{Th5.3}.
\end{proof}

 \medskip
\noindent
{\bf Acknowledgements}:
P. Chen was supported by NNSF of China 11501583, Guangdong Natural Science Foundation
2016A030313351.
X. Duong was supported by the Australian Research Council (ARC) through the research
grant DP190100970. J. Li was supported by the Australian Research Council (ARC) through the
research grant DP170101060 and by Macquarie University Research Seeding Grant.
L. Yan was supported by the NNSF of China, Grant
No. 11521101 and 11871480, and by the Australian Research Council (ARC) through the research
grant DP190100970.
 We would like to thank T. A. Bui, Z. Fan, E.M. Ouhabaz, A. Raich, A. Sikora and L. Song
  for helpful discussions.

 \end{document}